\documentclass[12pt,reqno]{amsart}

\usepackage[english]{babel}
\usepackage[utf8]{inputenc}
\usepackage[T1]{fontenc}

\usepackage{bbm}
\usepackage[nocompress]{cite}
\usepackage{float, verbatim}
\usepackage{amsmath}
\usepackage{amssymb}
\usepackage{amsthm}
\usepackage{amsfonts}
\usepackage{graphicx}
\usepackage{mathtools}
\usepackage[pagebackref]{hyperref}
\usepackage{tikz}
\usepackage{xcolor}
\usepackage{pgfplots}
\usepackage{caption}
\usepackage{subcaption}
\usepackage{enumitem}
\usepackage{url}
\usepackage{dsfont}

\hypersetup{
	colorlinks=true,
	linkcolor=blue,
	filecolor=magenta,      
	urlcolor=cyan,
	citecolor= red,
}

\usepackage[colorinlistoftodos]{todonotes}

\pgfplotsset{compat=1.18}

\newcommand{\boxd}{\dim_{\mathrm{B}}}

\newcommand{\Haus}{\dim_{\mathrm{H}}}

\renewcommand{\epsilon}{\varepsilon}

\newtheorem*{thm*}{Theorem}
\newtheorem*{conj*}{Conjecture}
\newtheorem*{ques*}{Question}
\newtheorem*{rem*}{Remark}
\newtheorem*{defn*}{Definition}
\newtheorem*{mainques*}{Main questions}

\newtheorem{thm}{Theorem}[section]

\newtheorem{lma}[thm]{Lemma}

\newtheorem{defn}[thm]{Definition}

\newtheorem{prop}[thm]{Proposition}
\newtheorem{conj}[thm]{Conjecture}
\newtheorem{claim}[thm]{Claim}
\newtheorem{rem}[thm]{Remark}
\newtheorem{ques}[thm]{Question}

\newtheorem{exm}[thm]{Example}

\def\RR{\mathbb{R}}
\def\CC{\mathbb{C}}

\def\ZZ{\mathbb{Z}}
\def\NN{\mathbb{N}}

\def\supp{\mathrm{supp}}

\def \bxi {{\boldsymbol{\xi}}}
\def \bzero {{\mathbf{0}}}

\def \bN {{\mathbb N}}

\def \cM {{\mathcal{M}}}

\def \epsilon {{\varepsilon}}
\def \lam {{\lambda}}

\def \bt {{\mathbf{t}}}

\def \bzero {{\boldsymbol{0}}}

\def \bC {\mathbb C}

\def \bN {\mathbb N}

\def \bS {\mathbb S}

\def \ba {\mathbf a}
\def \bb {\mathbf b}
\def \bc {\mathbf c}

\def \bh {\mathbf h}

\def \bn {\mathbf n}

\def \bt {\mathbf t}
\def \bu {\mathbf u}
\def \bv {\mathbf v}
\def \bx {\mathbf x}

\def \by {\mathbf y}
\def \bz {\mathbf z}

\def \bzero {\mathbf 0}

\def \bxi {{\boldsymbol{\xi}}}

\def \cA {\mathcal A}

\def \cD {\mathcal D}

\def \cF {\mathcal F}

\def \cM {\mathcal M}

\def \cP {\mathcal P}

\def \leq {\leqslant}

\def \geq {\geqslant}

\def \dim {\mathrm{dim}}

\def \det {\mathrm{det}}

\def \supp {{\mathrm{supp}}}

\def \ds1 {\mathds{1}}

\def \lam {{\lambda}}

\def \epsilon {{\varepsilon}}

\numberwithin{equation}{section}

\title[Fourier decay of nonlinear images]{Fourier transform of nonlinear images of self-similar measures: qualitative aspects}

\author{Amlan Banaji}
\address{Amlan Banaji, Department of Mathematics and Statistics, University of Jyv\"askyl\"a, Agora, Mattilanniemi 2, Jyv\"askyl\"a, 40100, Finland}
\curraddr{}
\email{banajimath@gmail.com}

\author{Han Yu}
\address{Han Yu, College of Mathematics and Statistics, Center of Mathematics, Chongqing University, Chongqing, 401331, China.}
\curraddr{}
\email{han.yu.2@cqu.edu.cn}

\thanks{}

\subjclass[2020]{42A38 (primary), 28A80 (secondary)}
\keywords{Fourier decay, self-similar measures, self-conformal measures, nonlinear images, \L{}ojasiewicz inequality}

\begin{document}


\begin{abstract}
    The goal of this paper is to establish polynomial Fourier decay for images of self-similar measures $\mu$ on $\mathbb{R}^k$ under sufficiently nonlinear real-analytic maps $f \colon \mathbb{R}^k \to \mathbb{R}^d$. For example, we prove that if $f$ is analytic on $\mathbb{R}^k$, its graph does not lie in an affine hyperplane in $\mathbb{R}^{k+d}$, and $\mu$ is not supported in an affine hyperplane in $\mathbb{R}^k$, then the image measure has polynomial Fourier decay. Key steps in the proof include establishing a uniform {\L}ojasiewicz-type inequality for self-similar measures, and using the decay of the Fourier transform of $\mu$ outside a very small exceptional set of frequencies. As an application of our results, we prove polynomial Fourier decay for self-conformal measures on $\mathbb{C}$ for a large class of complex analytic IFSs which are not self-similar but are conjugate to a linear IFS via an analytic map. 
\end{abstract}

\maketitle

\tableofcontents

\section{Introduction}

\subsection{Fourier decay of nonlinear images}
In this paper, we consider polynomial Fourier decay for nonlinear images of self-similar measures. 
Given a compactly supported Borel probability measure $\nu$ on $\mathbb{R}^d$, its Fourier transform is 
\[  \widehat{\nu} \colon \RR^d \to \CC, \qquad \widehat{\nu}(\bxi) \coloneqq \int_{\RR^d} e^{-2 \pi i \langle \bxi,\bx \rangle} d\nu(\bx). \]
The problem is given as follows. 
Let $k,d$ be positive integers and let $f \colon \mathbb{R}^k\to\mathbb{R}^d$ be a real analytic map. 
Let $\mu$ be a self-similar measure on $\RR^k$, and consider the pushforward measure $\mu_f$ defined by $\mu_f(B) = \mu(f^{-1}(B))$ for Borel sets $B \subseteq \RR^d$. 
We ask the following question about $\widehat{\mu_f}$, whose answer (as we will see) depends on properties of $\mu$ and $f$. 
\begin{ques}\label{ques:main}
    Does $\mu_f$ have polynomial Fourier decay? In other words, are there some $\sigma, c>0$ such that 
    \[|\widehat{\mu_f}(\bxi)|\leq c|\bxi|^{-\sigma} \qquad \mbox{ for all } \bxi\in\mathbb{R}^d\setminus \{\mathbf{0}\}? \] 
\end{ques}
\begin{rem}
    If $\mu$ has a smooth compactly supported density on $\mathbb{R}^k$, then the polynomial Fourier decay property for $\mu_f$ holds if $f$ is analytic and non-degenerate on a neighbourhood of $\supp(\mu)$ (our definition of non-degeneracy will be made precise in Section~\ref{s:nondeg}). 
    This is a classical result in harmonic analysis, see \cite[VIII, Section~3.2]{Ste1993}. 
\end{rem}
Question~\ref{ques:main} is motivated by work of Kaufman~\cite{K82}, who observed that the image measure can have polynomial Fourier decay even if the self-similar measure is a non-Rajchman measure (meaning that $\limsup_{|\bxi| \to \infty} |\widehat{\mu}(\bxi)| > 0$) like the Cantor--Lebesgue measure. 
We defer further historical background to a later section. 
Our answer to Question~\ref{ques:main} is given by the following theorem, which is a main result of the paper. 
To avoid trivialities, we make the standing assumption throughout the paper that none of our self-similar measures are just a single atom. 

\begin{thm}\label{thm: main}
    Let $k,d \geq 1$ be integers, let $\mu$ be a self-similar measure on $\mathbb{R}^k$, and let $f\colon \mathbb{R}^k\to\mathbb{R}^d$ be a real-analytic map. 
    Let $H \subseteq \RR^k$ be the smallest affine subspace of $\RR^k$ containing the support of $\mu$, and assume that $f(H)$ is not contained in an affine hyperplane of $\RR^d$.\footnote{We say $H \subseteq \RR^k$ is an \emph{affine subspace} of $\RR^k$ if there exist $j \in \{0,\dotsc,k\}$, orthonormal vectors $\mathbf{e}_1,\dotsc,\mathbf{e}_j \in \RR^k$, and $\mathbf{t} \in \RR^k$ such that $H = \{ \bx \in \RR^k : \langle \bx - \mathbf{t},\mathbf{e}_i \rangle = 0 \, \forall i \in \{1,\dotsc,j \} \}$. In the $j=1$ case, we call $H$ an \emph{affine hyperplane}. In the $j=0$ case, the affine subspace $H = \RR^k$.}
    Then $\mu_f$ has polynomial Fourier decay unless both of the following conditions hold: 
    \begin{enumerate}
        \item\label{i:degenerate} The restriction $f|_{H}$ is partially linear in the sense that there are real numbers $b_1, \dotsc, b_d$, not all zero, such that for $f=(f_1,\dotsc,f_d)$, the function $\sum_i b_i f_i$ restricts to an affine map $H\to \RR$. 
        
        \item\label{i:nopoly} If $\nu$ is a self-similar measure on $\RR^{\dim H}$ and $\iota \colon \RR^{\dim H} \hookrightarrow H$ is an isometric embedding with $\mu = \nu_{\iota}$, then $\nu$ does \emph{not} have polynomial Fourier decay.\footnote{Whether $\nu$ has polynomial Fourier decay is clearly independent of the choice of $\nu$ and $\iota$.}
    \end{enumerate}
\end{thm}

We will see in Section~\ref{s:nondeg} that the partial linearity conclusion from Theorem~\ref{thm: main}~\eqref{i:degenerate} is equivalent to the graph of $f|_H$ being contained in an affine subspace of $H \times \RR^{d}$ which is proper in the sense that its dimension is strictly less than $d + \dim H$. 

The proof of Theorem~\ref{thm: main} contains substantial additional difficulties compared to previous work in the $d=1$ or $k=1$ cases. In particular, we need to establish a uniform \L{}ojasiewicz-type inequality for self-similar measures, see Section~\ref{ss:lojafirst}. Moreover, we need a recent result of Khalil~\cite{Khalil} which can be used to show that self-similar measures on $\RR^k$ which are not supported in an affine hyperplane in $\RR^k$ have polynomial Fourier decay outside a very sparse set of exceptional frequencies, see Section~\ref{ss:afdfirst}. 

Theorem~\ref{thm: main} is essentially sharp in the sense indicated by the following simple observation. 
\begin{prop}\label{prop: converse}
    Let $f \colon \RR^k \to \RR^d$ be an analytic map which is partially linear on $\RR^k$. 
    Then there is some self-similar measure $\mu$ on $\RR^k$ which is not supported in any affine hyperplane in $\RR^k$, and such that neither $\mu$ nor $\mu_f$ is Rajchman. 
\end{prop}
\begin{proof}
There exist $b_1,\dotsc,b_d$, not all zero, such that $\sum_i b_i f_i$ is an affine function. 
Next, observe that
\[
\widehat{\mu_f} (\bxi) = \int e^{-2\pi i \sum_{i}\xi_i f_i(x)}d\mu(x).
\]
Then if $\bxi$ is chosen to be along the line with direction $(b_1,\dotsc,b_d)$ we see that $\sum_i \xi_i f_i(x)=t_\bxi L(x)$ for some affine form $L$ and some real number $t_\bxi$ whose norm is proportional to the length of $\bxi$. 
We write this affine form as
\[
L(x)=\sum_{i=1}^{k} a_i x_i+b
\]
for real numbers $a_1,\dotsc,a_k,b$, and assume that this affine form is non-constant (or else the non-decay property is trivial). 
Observe that for any self-similar measure $\mu$ on $\RR^k$ and $\bxi$ in the direction of $(b_1,\dotsc,b_d)$, 
\begin{equation}\label{e:degnodecay}
\left| \int e^{-2\pi i \sum_{i}\xi_i f_i(x)}d\mu(x) \right| = \left|\int e^{-2\pi i t_\bxi\sum_i a_i x_i}d\mu(x)\right|=|\widehat{\mu}(t_\bxi (a_1,\dotsc,a_k))|.
\end{equation}
Let $\mu$ be the $k$-fold product of the Cantor--Lebesgue measure, rotated so that $\widehat{\mu}$ does not decay along the direction $(a_1,\dotsc,a_k)$. 
Then from~\eqref{e:degnodecay} we see that $|\widehat{\mu_f}|$ does not decay along $(b_1,\dotsc,b_d)$. 
\end{proof}

\begin{rem}\label{rem:notstrictconverse}
Proposition~\ref{prop: converse} shows that the partial linearity exception~\eqref{i:degenerate} in Theorem~\ref{thm: main} cannot be removed uniformly over the self-similar measures. Indeed, if $f\colon\RR^k\to\RR^d$ and a positive-dimensional affine subspace $H\subseteq\RR^k$ satisfy~\eqref{i:degenerate}, then applying the proposition to $f|_H\circ\iota$ gives some self-similar measure $\mu$ supported inside $H$ such that~\eqref{i:nopoly} holds and $\mu_f$ is non-Rajchman. 

However, there are also pairs $(\mu,f)$, satisfying the exceptional conditions~\eqref{i:degenerate} and~\eqref{i:nopoly}, for which $\mu_f$ \emph{does} have polynomial Fourier decay. For example, if $\mu=\nu\times\lam_{[0,1]}$ and $f(x,y)=y$, with $\nu$ being the middle-third Cantor measure and $\lam_{[0,1]}$ normalised Lebesgue measure on $[0,1]$, then~\eqref{i:degenerate} and~\eqref{i:nopoly} hold but $\mu_f$ has polynomial Fourier decay.\footnote{A similar example was identified by Ethan Eglite for a gap in the statement of Theorem~\ref{thm: main} of an earlier version of this paper.}
\end{rem}

One of the applications of our results relates to Fourier decay for self-conformal measures in the plane. 
A headline result is Theorem~\ref{thm: analyticconformal}. 
In Section~\ref{s:conformal}, we will contextualise the results and describe precisely what we mean by all the terms in the following statement. 

\begin{thm}\label{thm: analyticconformal}
    Let $\nu$ be a self-conformal measure for an IFS of holomorphic contractions on a ball in $\CC$. Assume that $\nu$ is not self-similar and not supported inside an analytic curve.
    Then $\nu$ has polynomial Fourier decay. 
\end{thm}

In fact, we prove Theorem~\ref{thm: analyticconformal} in the case when the conformal IFS can be conjugated to a self-similar IFS by a holomorphic diffeomorphism. 
The non-conjugate case is a direct consequence of a recent result of Algom, Rodriguez~Hertz and Wang~\cite{AHW25}. 

We emphasise that in our results, in particular Theorems~\ref{thm: main} and~\ref{thm: analyticconformal}, we do not assume that the IFS under consideration satisfies any separation conditions. 

\subsection*{Structure of the paper}
Section~\ref{ss:preliminaries} gives necessary preliminaries and references to relevant background literature. 
In Section~\ref{s:ideas} we state Theorem~\ref{thm: maingeneraldomain}, which is a stronger result than Theorem~\ref{thm: main} as it does not require the domain of the analytic map to be the whole of $\RR^k$. Theorem~\ref{thm: maingeneraldomain} shows that various combinations of conditions for the self-similar measure and analytic map guarantee polynomial Fourier decay for the pushforward. 

Section~\ref{s:proofassumingLoandAFD} deduces Theorem~\ref{thm: maingeneraldomain} from Lemma~\ref{lma: startingdecomp} and two additional ingredients, namely Theorems~\ref{thm: Lo} and~\ref{thm: AFD}. 
Lemma~\ref{lma: startingdecomp} approximates the Fourier transform of the image measure by an average of the Fourier transform of pieces of the self-similar measure across many frequencies; this was already proved as Lemmas~3.2 and~3.3 in~\cite{BY-quantitative}. 
Theorem~\ref{thm: Lo} establishes a \L{}ojasiewicz-type inequality for self-similar measures and is proved in Section~\ref{s:lo}. 
Theorem~\ref{thm: AFD} establishes Fourier decay of self-similar measures outside a sparse set of frequencies and is proved in Section~\ref{s:afd}. 
Theorems~\ref{thm: lifttograph NE}, \ref{thm: lifttograph - expanding} and~\ref{thm: polygraph} are also proved in Section~\ref{s:proofassumingLoandAFD}; these results consider Fourier decay of lifts of self-similar measures to graphs of nonlinear maps. 
Section~\ref{s:conformal} applies our results to self-conformal measures, proving Theorem~\ref{thm: analyticconformal} using Theorem~\ref{thm: maingeneraldomain}. 
Finally, Section~\ref{sec: open} presents several possible avenues for future research. See Figure \ref{fig:structure} for a schematic representation of the paper. 

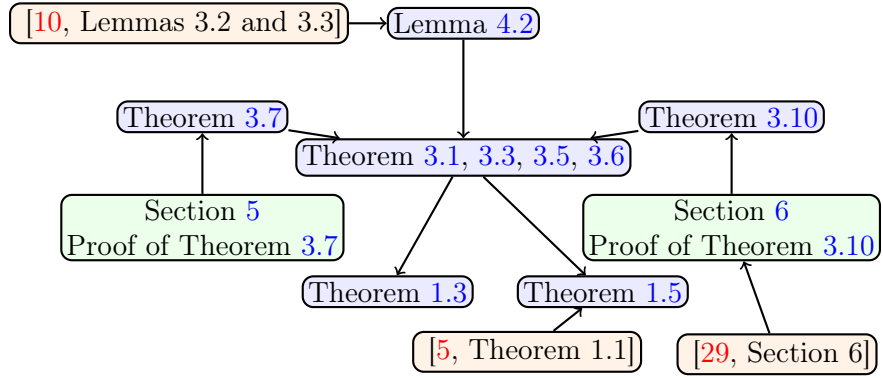
\begin{figure}[htbp]
\centering

\begin{tikzpicture}[
    node distance=0.8cm and 0.8cm,
    box/.style={
        draw,
        rounded corners,
        thick,
        align=center,
        font=\small,
        inner sep=2pt,
        minimum height=0.40cm
    },
    result/.style={box, fill=blue!8},
    section/.style={box, fill=green!8},
    source/.style={box,fill=orange!10},
    arrow/.style={->, thick}]

\node[result] (lem41) {Lemma~\ref{lma: startingdecomp}};

\node[source] (lemma8) [left=of lem41, xshift=0.3cm] {\cite[Lemmas 3.2 and 3.3]{BY-quantitative}};

\node[result] (thm37) [below left=of lem41, xshift=-0.5cm] {Theorem~\ref{thm: Lo}};

\node[result] (thm38) [below right=of lem41, xshift=0.5cm] {Theorem~\ref{thm: AFD}};

\node[section] (sec6) [below=of thm37] {Section~\ref{s:lo}\\Proof of Theorem~\ref{thm: Lo}};

\node[section] (sec5) [below=of thm38] {Section~\ref{s:afd}\\Proof of Theorem~\ref{thm: AFD}};

\node[result, minimum width=4.4cm] (thmmain) [below=1.3cm of lem41] {Theorems~\ref{thm: maingeneraldomain}, \ref{thm: lifttograph NE}, \ref{thm: lifttograph - expanding}, \ref{thm: polygraph}};

\node[result] (intro) [below=of thmmain, xshift=-1.0cm, yshift=-0.5cm] {Theorem~\ref{thm: main}};

\node[result] (thm15) [below left=of thmmain, xshift=6.0cm, yshift=-0.5cm] {Theorem~\ref{thm: analyticconformal}};

\node[source] (AHW25) [below = of thm15, yshift=0.5cm, xshift=-1cm]{\cite[Theorem~1.1]{AHW25}};

\node[source] (Khalil) [below= of sec5, yshift=-0.15cm, xshift=0.6cm] {\cite[Section~6]{Khalil}};

\draw[arrow] (Khalil) -- (sec5);

\draw[arrow] (AHW25) -- (thm15);

\draw[arrow] (lemma8) -- (lem41);
\draw[arrow] (sec6) -- (thm37);
\draw[arrow] (sec5) -- (thm38);

\draw[arrow] (lem41) -- (thmmain);
\draw[arrow] (thm37) -- (thmmain);
\draw[arrow] (thm38) -- (thmmain);

\draw[arrow] (thmmain) -- (intro);

\draw[arrow] (thmmain) -- (thm15);

\end{tikzpicture}

\caption{Structure of the paper.}
\label{fig:structure}
\end{figure}

\section{Preliminaries and background literature}\label{ss:preliminaries}

\subsection{General notation}
We use Vinogradov and Bachmann--Landau notation: given real or complex-valued functions $f,g$ we write $f \ll g$ or $f = O(g)$ to mean $|f| \leq C|g|$ pointwise for some constant $C > 0$, and write $f \asymp g$ if $f \ll g$ and $g \ll f$. Subscripts may indicate parameters that the implicit constants are allowed to depend on. 
Given $\bx \in \RR^k$ and $r > 0$, the open ball of radius $r>0$ centred at $\bx$ is denoted $B(\bx,r) = \{\by \in \RR^k : |\by-\bx| < r\}$. 
For a set $W \subset \RR^k$ and $r > 0$ we denote the $r$-neighbourhood by 
\[ W^{(r)} \coloneqq \{ y \in \RR^k : |x-y| < r \mbox{ for some } x \in W\}. \]

\subsection{Real analytic maps and their properties}\label{s:nondeg}
We explain some terminology that we use. Let $k,d \geq 1$ be integers. 
\begin{defn}
    Given a non-empty set $K \subseteq \RR^k$, we say that $f = (f_1,\dotsc,f_d) \colon K \to \RR^d$ is \emph{(real) analytic} if there exists an open neighbourhood $U \subseteq \RR^k$ of $K$ and an extension $\widetilde f \colon U \to \RR^d$ of $f$ such that for all $\bx = (x_1,\dotsc,x_k) \in U$ and $1 \leq j \leq d$, there is a power series in $x_1,\dotsc,x_k$ which converges to $\widetilde f_j$ in an open neighbourhood of $\bx$.
\end{defn}
\vspace{0.2cm}
Now let $U \subseteq \RR^k$ be open, and let $f \colon U \to \RR^d$ be analytic. 
Let $\bv\in\mathbb{S}^{d-1}$. Define $f_\bv \colon U \to \RR$ by $f_{\bv}(\bx) \coloneqq \sum_{j=1}^{d}v_j f_j(\bx)$. Define 
\begin{equation}\label{e:definepv}
P_\bv \colon U \to\mathbb{R}^k, \qquad P_\bv(\bx) \coloneqq \nabla f_\bv(\bx) = (\nabla f (\bx))^T (\bv). 
\end{equation}
We define several different assumptions on our function $f$ that we will use. 
\begin{defn}
Let $f \colon U \to \RR^d$ be analytic. 
\begin{enumerate}
    \item We say that $f$ is \emph{non-degenerate} if for all balls $B \subset U$ and all $\bv \in \mathbb{S}^{d-1}$, the corresponding $P_\bv$ function is not constant on $B \cap U$. 
    \item We say that $f$ is \emph{non-conical} if for all balls $B \subset U$ and all $\bv \in \mathbb{S}^{d-1}$, $\bx \mapsto |P_\bv(\bx)|$ is not a constant function on $B \cap U$. 
    \item We say that $f$ is \emph{non-trapped} if for all balls $B \subset U$, $f(B \cap U)$ is not contained in an affine hyperplane in $\RR^d$. 
\end{enumerate}
\end{defn}
We make the following observations about the relations between these conditions. 
\begin{rem}\label{rem:degimplications}
    \begin{enumerate}
        \item Non-conical maps are non-degenerate, and the two conditions are equivalent when $k=1$. 
        \item If $f$ is non-degenerate then $f$ is clearly non-affine and non-trapped. 
        \item The map $f \colon \RR^k \setminus \{ \mathbf{0} \} \to \RR$ ($k > 1$), $\bx \mapsto |\bx|$ (whose graph is a cone) is an example of a non-degenerate but conical function. 
        \item\label{i:liftdeg} The map $f \colon \RR \to \RR^2$, $x \mapsto (x,(x+1)^2)$ is an example of a non-affine and non-trapped analytic map that is degenerate. 
    \end{enumerate}
\end{rem}

The following proposition relates our non-degeneracy to a more familiar one for analytic functions, see Kleinbock and Margulis~\cite{KM98}. 
\begin{prop}\label{prop: degequivalences}
    Given an analytic map $f \colon U \to \RR^d$, the following are equivalent: 
    \begin{enumerate}
        \item\label{i:deg} $f$ is degenerate, 
        \item\label{i:partiallin} $f$ is partially linear on some ball $B \subset U$ in the sense that there are $b_1,\dotsc,b_d \in \RR$, not all $0$, such that $\sum_i b_i f_i$ restricts to an affine map $B \cap U \to \RR$. 
        \item\label{i:graph} For some ball $B \subset U$, the graph of $f|_{B \cap U}$ is contained in an affine hyperplane in $\RR^{k+d}$. 
    \end{enumerate}
\end{prop}
\begin{proof}
    \eqref{i:deg} is equivalent to the existence of a ball $B$ and $\bv\in\bS^{d-1}$ such that $P_\bv$ is constant on $B \cap U$ (or equivalently $f_\bv$ is affine on $B \cap U$). 
    This implies that $(v_1 f_1+\dots+v_d f_d)(\bx) = L(\bx)$ on $B \cap U$ for some affine form $L$, which gives~\eqref{i:partiallin}. 
    Now~\eqref{i:partiallin} implies that $(\bx,f_1,\dotsc,f_d)$ satisfies some nontrivial affine form on $B \cap U$, which is equivalent to~\eqref{i:graph}. 
    
    Conversely,~\eqref{i:graph} implies that there are numbers $a_1,\dotsc,a_k,b_1,\dotsc,b_d,b$, not all $0$, such that
    \[
    \sum^{k}_{i=1} a_i x_i+\sum^d_{i=1} b_i f_i(\bx) + b = 0
    \]
    for all $\bx = (x_1,\dotsc,x_k) \in B \cap U$. 
    Not all the $b_i$ can be zero otherwise we would have a linear condition which cannot be satisfied by all values of the free variables $x_1,\dotsc,x_k$. Therefore we see~\eqref{i:partiallin} holds. 
    Now~\eqref{i:partiallin} implies that $P_{(b_1,\dotsc,b_d)/||(b_1,\dotsc,b_d)||}$ is constant on $B \cap U$, giving~\eqref{i:deg}. 
\end{proof}

Note that our example $\bx \mapsto |\bx|$ for a non-degenerate conical function is not globally analytic on $\RR^k$. Indeed, the following result holds. 
\begin{prop}\label{prop: globaldegcon}
    Let $f \colon \mathbb{R}^k \to \RR^d$ be analytic on $\RR^k$. 
    Then $f$ is conical if and only if $f$ degenerates.
\end{prop}
\begin{proof}
    We only need to prove the forward implication. 
    Suppose $|P_\bv|$ is constant over some non-empty open set for some $\bv$. Then $|P_{\bv}|^2$, and hence $|P_\bv|$, is globally constant. Then $|\nabla f_\bv|$ is constant. It is well known from the theory of the eikonal equation that $f_\bv$ must be affine, see for example~\cite{Sakai}. But $f_\bv$ being affine implies that $P_\bv \colon \bx \mapsto \nabla f_\bv(\bx)$ is a constant map, hence the degeneracy. 
\end{proof}

For our uniform measure \L{}ojasiewicz inequality we will need the following definition. 
\begin{defn}\label{d:compact}
    Let $U \subseteq \RR^k$ be a connected open set and consider a family $\mathcal{F}$ of analytic maps from $U \to \RR^d$. 
    We say that $\mathcal{F}$ is \emph{(strongly) compact}\footnote{This is the notion of compactness we will use. Hence we drop the word ``strongly''.} if it is
    
    \begin{itemize}
    \item uniformly analytic in the sense that over each compact $K\subset U$, the convergence radius for each $f\in\cF$ is bounded away from zero. In other words, for some $r>0$, for any $x\in K$ and $f\in \cF$, the convergence radius of the Taylor expansion of $f$ at $x$ is at least $r$. In particular, there is an open complex neighbourhood $V_K\subset\bC^k$ of $K$, which can be chosen to be a finite union of polydiscs, such that each $f\in\cF$ admits a holomorphic extension $\widetilde f\colon V_K\to\bC^d$.
    \item compact in the complex-compact-open topology, in the sense that for each compact $K\subset U$, the neighbourhood $V_K$ and the extensions above can be chosen so that $\{\widetilde f:f\in\cF\}$ is compact in the compact-open topology on $V_K$. In other words, every sequence in $\cF$ contains a subsequence whose extensions converge uniformly on compact subsets of this fixed $V_K$ to the extension of an element of $\cF$.
        
    \end{itemize}
\end{defn}
The usual compactness in the compact-open topology does not imply the strong notion above. For example, consider
\[
f_n \colon x\mapsto x(1-e^{-n^2x^2})^n, \quad \cF=\{x\mapsto x\}\cup \{f_n\}_{n\geq 1}.
\]
Near $0$, no unions of polydiscs can be found to maintain the compactness of the family in the complex-compact-open topology. On the other hand, an example of a compact family is a collection of polynomials with bounded degree and a compact coefficient set.
\subsection{Self-similar sets and measures}\label{ss:ifs}
    For general fractal geometry background we refer the reader to the books~\cite{FaTechniques,Ma1,Ma2,BSSbook}. 
    Let $k\geq 1$ and $N>1$ be integers, and let $D \subset \RR^k$ be compact. 
    Let $\Lambda = \{f_i \colon D \to D\}_{1 \leq i \leq N}$ be contraction maps (i.e. $\rho$-Lipschitz maps for some $\rho < 1$); $\Lambda$ is called an \emph{iterated function system} or \emph{IFS} for short. 
    By Hutchinson's theorem~\cite{Hutchinson}, there is a unique non-empty compact set $K \subset D$, called the \emph{attractor}, and for $p_1,\dotsc,p_N\in (0,1)$ with $\sum_i p_i=1$ there is a unique Borel probability measure $\mu$ whose support equals $K$, such that 
    \begin{equation*}
    K=\bigcup_{i}f_i(K), \qquad \mu=\sum_{i} p_i \mu_{f_i} 
    \end{equation*}
    (recalling that $\mu_{f_i}$ denotes the pushforward of $\mu$ by $f_i$). 
    We will always assume that the contractions do not share a common fixed point. 

    If there exist $r_1,\dotsc,r_N\in (0,1)$ and $O_1,\dotsc,O_N\in O_k(\mathbb{R})$ (possibly reflected) rotations and $\bt_1,\dotsc,\bt_N\in\mathbb{R}^k$ such that $\Lambda$ consists of the similarity maps
    \[
    \Lambda=\{f_i(\cdot)=r_iO_i(\cdot)+\bt_i\}_{i\in\{1,\dotsc,N\}},
    \]
    then we say that $\Lambda, K, \mu$ are self-similar. 
    In this case, we say that $\Lambda,K,\mu$ are homogeneous if $r_i O_i$ are all the same for $i\in\{1,\dotsc,N\}$. 

    Now let $\mu$ be a self-similar measure on $\RR^k$. 
Without loss of generality, we can assume that the support of $\mu$ is contained in a unit cube and has diameter at most $1$. 
Let $\Lambda=\{f_i\}_{i=1}^N$ be an IFS generating $\mu$, with contraction ratios $r_i$ and probability weights $p_i$. Write $\rho_m=\min_i r_i$. For a finite word $\omega=(i_1,\dotsc,i_l)$, write $f_\omega=f_{i_1}\circ\dotsb\circ f_{i_l}$, $r_\omega=r_{i_1}\dotsb r_{i_l}$, $p_\omega=p_{i_1}\dotsb p_{i_l}$, and let $\omega^-$ denote the word with its last symbol removed, with $r_{\varnothing}=1$. Given an infinite sequence $i_1i_2\dots\in\{1,\dots,N\}^\bN$, we define
\[
\pi(i_1 i_2\dots) \coloneqq \lim_{l\to\infty}f_{i_1}\circ\dots\circ f_{i_l}(\bzero)\in K.
\]
The choice of $\bzero$ has no significance; it can be replaced with any point. The limit exists and does not depend on the choice of that point. We call $\pi\colon \{1,\dots,N\}^\bN\to K$ the coding map. In this way, self-similar measures are pushforwards of Bernoulli measures on $\{1,\dots,N\}^\bN$ under $\pi$. 

For $n\geq 1$, consider the finite collection of stopping words
\[
\mathcal{W}_n\coloneqq\{\omega:r_\omega\leq 2^{-n}<r_{\omega^-}\}.
\]
Every infinite word has exactly one prefix in $\mathcal{W}_n$. Thus, writing $\mu_{\omega,n}\coloneqq p_\omega\mu_{f_\omega}$ for $\omega\in\mathcal{W}_n$, we have
\[
\mu=\sum_{\omega\in\mathcal{W}_n}\mu_{\omega,n},\qquad |\mu_{\omega,n}|=p_\omega.
\]
We call $\mu_{\omega,n}$ a \emph{$2^{-n}$-branch} of $\mu$. Its support is a similar copy of $K=\supp(\mu)$ with contraction ratio in $(\rho_m2^{-n},2^{-n}]$, and hence has diameter in $(\rho_m2^{-n}\operatorname{diam}K,2^{-n}\operatorname{diam}K]$.
Each such branch intersects at least one and at most $3^k$ dyadic cubes in 
\[ \cD_n \coloneqq \{2^{-n}(\mathbf m+[0,1)^k):\mathbf m\in\mathbb Z^k\}.\] 
We choose one such cube and associate it with $\mu_{\omega,n}$. For each $D_n\in\cD_n$, let $\Lambda_{D_n}$ be the collection of stopping words associated with $D_n$, and write $\mu_{\omega,n}\sim D_n$. Then $\{\Lambda_{D_n}\}_{D_n\in\cD_n}$ is a partition of $\mathcal{W}_n$, inducing the corresponding decomposition of $\mu$. For each $\omega\in\Lambda_{D_n}$, the branch $\mu_{\omega,n}$ is supported in $3D_n$, the tripling of $D_n$ with the same centre.

\subsection{Affine irreducibility}
\begin{defn}
We say that a Borel probability measure $\mu$ on $\mathbb{R}^k$ is \emph{affinely irreducible} if $\mu(H)=0$ for every affine hyperplane $H\subset\mathbb{R}^k$. 
\end{defn}
For example, if $k=1$, then every self-similar measure whose support is not a finite set must be affinely irreducible. As another example, notice that normalised Lebesgue measures on bounded open sets are affinely irreducible. This example extends to natural surface measures on non-degenerate manifolds.

\subsection{Expanding systems}
\begin{defn}
Let $G$ be a group, and $F\subset G$ a finite set. 
We say that $G$ is \emph{non-expanding} with respect to $F$ (or $F$ is \emph{non-expanding}) if for each $\epsilon>0$,
\[
\# F_n\ll e^{\epsilon n},
\]
where $F_n\subset G$ is the collection of $f_1 \dotsb f_n$ for $f_1,\dotsc,f_n\in F$. 
\end{defn}
The Tits alternative~\cite{OS95,Tits} gives a more algebraic description of the expanding property for finitely generated linear groups. 

\begin{defn}[non-expanding self-similar system]
Let $\Lambda$ be a self-similar IFS. We say that it is \emph{non-expanding} if the collection of orthogonal parts $\{O_i\}$ is non-expanding when viewed as a subset of the Euclidean group on $\mathbb{R}^k$. 
We also say that a self-similar set/measure is non-expanding if there exists a non-expanding self-similar IFS generating the set/measure. 
\end{defn}
For example, every self-similar system with abelian rotation group is non-expanding. 
In particular, every self-similar system in $\mathbb{R}$ or $\mathbb{R}^2$ is non-expanding.

\subsection{Fourier decay of nonlinear images of self-similar measures}
The study of nonlinear images of self-similar measures was initiated by Kaufman~\cite{K82} for Bernoulli convolutions on the real line and later extended by~\cite{Tsujii,MS18,MosqueraOlivo} to homogeneous self-similar measures. 
Finally, the problem for $f\colon \mathbb{R}\to\mathbb{R}$ was completed for general self-similar measures by~\cite{ACWW25,BB25}. 
Recently, in~\cite{BY-quantitative} we quantified (explicitly lower-bounded the decay exponent of) all of the previous results and moreover considered the more general case $f\colon \mathbb{R}^k\to\mathbb{R}$. For vector-valued functions, \cite{AK} considered the case $f\colon \mathbb{R}\to\mathbb{R}^d$ where $d\geq 1$ can be arbitrary. 
The purpose of the present paper is to provide sufficient conditions for polynomial Fourier decay of analytic images $f\colon\mathbb{R}^k\to\mathbb{R}^d$, where $k,d\geq1$ are arbitrary, together with counterexamples showing that the partial linearity exception cannot be removed uniformly. There are still some `boundary cases' that we do not cover when the domain of $f$ is not the whole of $\RR^k$, but these cases are rather special, see Section~\ref{sec: open}. 

\subsection{Self-conformal measures}\label{ss:selfconfprelim}
Self-conformal measures are a natural generalisation of self-similar measures. 
Given a set $D \subset \RR^k$, we call a map $f \colon D \to \RR^k$ \emph{conformal} if on some open neighbourhood $U$ of $D$ it is $C^{1+\alpha}$ and preserves the angle and orientation of directed curves through each point in $U$. Equivalently, it is $C^{1+\alpha}$ on $U$ and the Jacobian at each point is a positive scalar multiplied by a rotation matrix. 
When $k=2$, conformal maps are precisely holomorphic maps with a non-vanishing complex derivative at every point in $U$. When $k\geq 3$, conformal maps have a very restricted form by Liouville's theorem and are in fact M\"obius transformations: they can be written as a composition of translations, rotations, reflections, and inversions in $(k-1)$-spheres. Note in particular that when $k \geq 2$, conformal maps are real analytic. 

Next, we introduce self-conformal measures. For simplicity, we do not attempt to give the most general possible setup. 
Let $\Omega \subset D \subset U \subset \CC$ where $\Omega$ is open, connected and non-empty, $D = \overline{\Omega}$ (i.e. $D$ is the topological closure of $\Omega$), and $U$ is open, connected and bounded. 
Let $\varphi_1,\dotsc,\varphi_N \colon U \to \CC$ be injective conformal maps satisfying $\overline{\varphi_i(U)} \subset U$ and $\varphi_i(D) \subset \Omega$ for each $i$. 
Assume moreover that 
\begin{align}\label{e:conformalcontractions}
\begin{split}
0 & < \inf\{||(D \varphi_i)_z|| : z \in U, 1 \leq i \leq N \} \\ &\leq \sup \{ ||(D \varphi_i)_z|| : z \in U, 1 \leq i \leq N \} < 1. 
\end{split}
\end{align}
Assume that $\varphi_1,\dotsc,\varphi_N$ do not share a common fixed point. 
We call $\{\varphi_1,\dotsc,\varphi_N\}$ a \emph{conformal IFS}. 
By~\eqref{e:conformalcontractions}, each $\varphi_i$ is a contraction with respect to the intrinsic length metric on $D$. 
Therefore by Hutchinson~\cite{Hutchinson}, there is a unique non-empty compact set $K \subseteq D$ called the \emph{self-conformal set} associated with this system, and for $p_1,\dotsc,p_N \in (0,1)$ with $\sum_i p_i = 1$ there is a unique Borel probability measure $\mu$ called the \emph{self-conformal measure}, such that
\begin{equation*}
    K=\bigcup_{i} \varphi_i(K), \qquad \mu=\sum_{i} p_i \mu_{\varphi_i}.
\end{equation*}
Note that $K$ is uncountable, $\mu$ is non-atomic, and the support of $\mu$ equals $K$. 
In this paper, whenever we talk about self-conformal IFSs, sets or measures, we will mean those which arise in this way. 
Most of our focus will be on the $k=2$ case, but we will also consider $k \geq 3$. 

For measures on the line (i.e. $k=1$), Fourier decay for nonlinear fractal measures (in particular self-conformal measures) has received a great deal of attention, see for instance \cite{AHW21,AHW22,KaufmanCtdFrac,QR,JordanSahlsten,BDfuchsian,LNP,SS,BakerKhalilSahlsten} and the survey~\cite{SahlstenSurvey}. 
It is known from a combination of results that self-conformal measures for an IFS of analytic contractions $\RR \to \RR$ which are not all affine have polynomial Fourier decay~\cite{AHW23,BS23,ACWW25,BB25}. 
In the $k=2$ case, polynomial Fourier decay of self-conformal measures which are not conjugate to linear has been studied in~\cite{AHW25} (see Section~\ref{ss:selfconfplane} for more details). Polynomial Fourier decay for self-conformal measures in $\RR^k$ for general $k$ has been considered in~\cite{BakerKhalilSahlsten}, under the strong separation condition and a uniform non-integrability assumption.

\section{Proof ideas and further results}\label{s:ideas}
\subsection{More refined pushforward results}\label{ss:subset}
In the introduction, we assumed the pushforward maps were analytic on the whole of $\RR^k$. However, we can also prove several results for maps which are only analytic on an open neighbourhood of the support of $\mu$. 
We will deduce Theorem~\ref{thm: main} easily from Theorem~\ref{thm: maingeneraldomain} in Section~\ref{ss:deduceheadlinefrommain}. 

\begin{thm}\label{thm: maingeneraldomain}
    Let $k,d \geq 1$ be integers, let $\mu$ be an affinely irreducible self-similar measure on $\RR^k$, let $U \subseteq \RR^k$ be an open set containing $\supp(\mu)$, and let $f \colon U \to \RR^d$ be real analytic. 
    Assume that at least one of the following three conditions holds: 
    \begin{enumerate}
        \item\label{i:expanding} $f$ is non-conical, or 
        \item\label{i:nonexpanding} $\mu$ is non-expanding and $f$ is non-degenerate, or 
        \item\label{i:selfsimhaspoly} $\mu$ has polynomial Fourier decay and $f$ is non-trapped.
    \end{enumerate}
    Then $\mu_f$ has polynomial Fourier decay. 
\end{thm}

We next make some comments about how sharp Theorem~\ref{thm: maingeneraldomain} is. 
\begin{rem}\label{rem:sharp}
    \begin{enumerate}
        \item\label{i:dreamconj} We are not presently able to prove that the conclusion of Theorem~\ref{thm: maingeneraldomain} holds when the non-conicality assumption~\eqref{i:expanding} is replaced by non-degeneracy (without further assumptions on $\mu$). 
        This is precisely \cite[Conjecture~5.10]{BY-quantitative} -- Theorem~\ref{thm: maingeneraldomain} makes very good progress towards this conjecture, for example Theorem~\ref{thm: maingeneraldomain}~\eqref{i:nonexpanding} resolves it affirmatively in the case $k \in \{1,2\}$. 
        \item The non-degeneracy assumption in~\eqref{i:nonexpanding} cannot be relaxed, because if $f$ were degenerate then by the proof of Proposition~\ref{prop: converse} there would be some non-Rajchman self-similar measure whose pushforward is non-Rajchman. 
        \item The non-trapped assumption in~\eqref{i:selfsimhaspoly} cannot be relaxed, because if $f(B \cap U)$ were contained in some affine hyperplane in $\RR^d$ (for some ball $B\subset U$) then the image of a self-similar measure supported inside $B$ would not have Fourier decay along the direction orthogonal to the hyperplane. In particular, we can take $\mu$ to be normalised Lebesgue measure on a cube contained in $B$; this is a self-similar measure with polynomial Fourier decay. 
    \end{enumerate}
\end{rem}

 Regarding Theorem~\ref{thm: maingeneraldomain}~\eqref{i:selfsimhaspoly}, we note that while some self-similar measures (such as the Cantor--Lebesgue measure) are non-Rajchman, on the line, self-similar measures `typically' have polynomial Fourier decay in some very strong sense~\cite{Solomyak}, and countably many specific examples with polynomial Fourier decay are known~\cite{DaiFengWang,Streck}. 

Our method of proving Theorem~\ref{thm: maingeneraldomain} can be used to show that lifts of self-similar measures to the graph of a nonlinear function have Fourier decay in all nontrivial directions. 
\begin{thm}\label{thm: lifttograph NE}
    Let $k,d \geq 1$ be integers and let $\mu$ be an affinely irreducible non-expanding self-similar measure on $\RR^k$. 
    Let $U \subseteq \RR^k$ be an open set containing $\supp(\mu)$, and let $f \colon U \to \RR^d$ be analytic and non-degenerate. 
    Define $T \colon U \to \RR^{k+d}$ by $\bx \mapsto (\bx,f(\bx))$. 
    Then there exists $\sigma > 0$ such that for all $\bxi \in \RR^{k+d}$ of unit length, with the last $d$ coordinates of $\bxi$ not all zero, there exists $c_{\bxi} > 0$ such that for all $t > 0$, 
    \[ |\widehat{\mu_T}(t \bxi)| \leq c_{\bxi} t^{-\sigma}. \]
\end{thm}
\begin{rem}
   The $k=1$ case of Theorem~\ref{thm: lifttograph NE} was established via a different method in \cite[Proposition~1.5]{AK}, see also Remark~\ref{rem: k or d=1 lo}~\eqref{i: k=1 lo}.
\end{rem}
For possibly expanding self-similar measures, we have the following result for the lift to the graph. 
\begin{thm}\label{thm: lifttograph - expanding}
    Let $k,d \geq 1$ be integers and let $\mu$ be an affinely irreducible self-similar measure on $\RR^k$. 
    Let $U \subseteq \RR^k$ be an open set containing $\supp(\mu)$, and let $f \colon U \to \RR^d$ be an analytic map which is not the sum of a conical function and an affine function. 
    Define $T \colon U \to \RR^{k+d}$ by $\bx \mapsto (\bx,f(\bx))$. 
    Fix $\varepsilon>0$. 
    Then there exist $\sigma,c > 0$ such that for all $\bxi \in \RR^{k+d}$ of unit length, with the last $d$ coordinates of $\bxi$ having a norm at least $\varepsilon$, for all $t > 0$, 
    \[ |\widehat{\mu_T}(t \bxi)| \leq c t^{-\sigma}. \]
\end{thm}

If the self-similar measure itself has polynomial Fourier decay, then the lift to the graph of non-degenerate analytic functions has polynomial Fourier decay in \emph{all} directions uniformly. 
\begin{thm}\label{thm: polygraph}
    Let $k,d \geq 1$ be integers and let $\mu$ be a self-similar measure on $\RR^k$ with polynomial Fourier decay. 
    Let $U \subseteq \RR^k$ be an open set containing $\supp(\mu)$, and let $f \colon U \to \RR^d$ be analytic and non-degenerate. 
    Define $T \colon U \to \RR^{k+d}$ by $\bx \mapsto (\bx,f(\bx))$. 
    Then $\mu_T$ has polynomial Fourier decay. 
\end{thm}

\subsection{Deducing Theorem \ref{thm: main}}\label{ss:deduceheadlinefrommain}
Theorem~\ref{thm: main} is a direct consequence of Theorem~\ref{thm: maingeneraldomain}~\eqref{i:expanding} and~\eqref{i:selfsimhaspoly}. 
\begin{proof}[Proof of Theorem \ref{thm: main} using Theorem \ref{thm: maingeneraldomain}]
    We restrict $f$ to $H$ and regard $H$ as the ambient space; then we can assume that $\mu$ is affinely irreducible. From Theorem~\ref{thm: maingeneraldomain}~\eqref{i:expanding} we can deduce the desired Fourier decay property unless $f$ is conical. 
    Since $f$ is globally real-analytic, recall from Proposition~\ref{prop: globaldegcon} that conicality is equivalent to degeneracy, which by Proposition~\ref{prop: degequivalences} is equivalent to partial linearity. 
    This finishes the proof of the Fourier decay assuming~\eqref{i:degenerate} does not hold. 
    Now if~\eqref{i:nopoly} does not hold then the desired Fourier decay follows immediately from Theorem~\ref{thm: maingeneraldomain}~\eqref{i:selfsimhaspoly}, since we assumed $f(H)$ was not contained in an affine hyperplane. 
\end{proof}

The proof of Theorem~\ref{thm: maingeneraldomain} is the most technical part of this paper. 
We first illustrate the ingredients which are of independent interest.

\subsection{A uniform measure \L ojasiewicz inequality}\label{ss:lojafirst}
Our next result, Theorem~\ref{thm: Lo}, is a non-concentration result which generalises a classical result of \L{}ojasiewicz~\cite{Lojasiewicz} which is well-known in real algebraic/analytic geometry~\cite{Colding}. 
Being unable to find an existing version of the \L{}ojasiewicz inequality that suits our needs, we will provide in this paper a standalone proof of the following result. Such a result may have other applications in future. 
Recently, in a paper on Diophantine approximation, B\'enard, He and Zhang used related results \cite[Theorem~4.4 and Corollary~4.5]{BHZ26} (for polynomials) to prove a Khintchine-type dichotomy for self-similar measures in $\mathbb{R}^k$. 
\begin{thm}\label{thm: Lo}
    Let $\cA$ be a compact family of real analytic functions from some connected open $U\subset\mathbb{R}^k$ to $\RR^d$. 
    Assume that $\cA$ does not contain the zero function. 
    Let $\mu$ be an affinely irreducible self-similar measure on $\mathbb{R}^k$ supported inside $U$. 
    Then for all $\varepsilon>0$, there are numbers $c,\beta>0$ such that for all $f\in \cA$, $\delta>0$,
    \[
    \mu(\{|f|<\delta\}^{(\delta^{\varepsilon})}) \leq c \delta^\beta, 
    \]
    where we recall $\{|f|<\delta\}^{(\delta^{\varepsilon})}$ denotes the $\delta^{\varepsilon}$-neighbourhood of $\{ \bx \in U : |f(\bx)|<\delta\}$. 
\end{thm}
\begin{rem}\label{rem: k or d=1 lo}
\begin{enumerate}
\item\label{i: k=1 lo} If $k=1$, then we have one-variable real-analytic functions. In this case, standard complex analysis can be used to establish this theorem. 
Namely, for $k=1$, Theorems~\ref{thm: maingeneraldomain}, \ref{thm: lifttograph NE} and~\ref{thm: lifttograph - expanding} can be proved via a simpler method. 
\item Even in the $d=1$ case, our proof of Theorems~\ref{thm: main} and~\ref{thm: maingeneraldomain} contains substantial additional ingredients compared to the $d=1$ results in~\cite{BY-quantitative}. This is because in~\cite{BY-quantitative} we assumed that the graph of the pushforward function has nonvanishing Gaussian curvature, which means that the zero sets of the functions in $\mathcal{A}$ in our application of Theorem~\ref{thm: Lo} consist of isolated points, so the proof of Theorem~\ref{thm: Lo} simplifies substantially. 
\end{enumerate}
\end{rem}
We deduce Theorem~\ref{thm: Lo} from a more quantitative statement, Proposition~\ref{p:quanifiedlo}, which is also used to prove Theorem~\ref{thm: lifttograph NE}. 
The requirement that $\cA$ be a compact analytic family on an open neighbourhood of $K=\supp(\mu)$ allows the consideration of analytic functions that are only defined around $K$ rather than globally. 
The $\delta^{\varepsilon}$ neighbourhood in the statement of Theorem~\ref{thm: Lo} will be convenient in the proof of Theorem~\ref{thm: maingeneraldomain}. 

Suppose $\cA$ is a finite collection and $\mu$ has a smooth compactly supported density. Let $K=\supp(\mu)$ and choose a compact neighbourhood $K'\subset U$ of $K$. For each $f\in\cA$ with zeros near $K$, the classical \L{}ojasiewicz inequality gives constants $C,\beta>0$ such that, for all sufficiently small $\delta>0$,
\[
\{\bx\in K:|f(\bx)|<\delta\}\subset \{\bx\in U:f(\bx)=0\}^{(C\delta^\beta)}.
\]
The zero set is a proper analytic variety, since $f$ is not identically zero, and its neighbourhoods have polynomially small Lebesgue measure near $K$. If $f$ has no zeros near $K$, including the case of a nonzero constant, the sublevel set is empty for sufficiently small $\delta$. The same estimates apply to the thickened sublevel sets: a bound for the first derivatives on $K'$ replaces $\delta$ by a constant multiple of $\delta+\delta^\varepsilon$. Taking the minimum of finitely many positive exponents and increasing the constants proves the conclusion in this case.
The point of Theorem~\ref{thm: Lo} is that $\mu$ is self-similar and the collection of functions may not be finite. The proof in Section~\ref{s:lo} establishes the required measure estimates and their uniformity.

\subsection{Fourier decay outside sparse frequencies}\label{ss:afdfirst}
We will also need the fact that affinely irreducible self-similar measures have Fourier decay outside of a very sparse set of frequencies. 
More precisely, consider the following property, which can be shown to be equivalent to an \emph{$L^2$-flattening} property (increase in $L^2$ dimension of iterated convolutions, as in \cite[(1.3)]{AK} for instance). 
\begin{defn}\label{d:afd}
Let $\mu$ be a probability measure on $\mathbb{R}^k$. We say that $\mu$ (or $\widehat{\mu}$) has \emph{Fourier decay outside sparse frequencies} if for each $\epsilon>0$ there exist $\delta,C>0$ so that
\[
|\{\bxi \in \RR^k:|\widehat{\mu}(\bxi)|\geq R^{-\delta}, |\bxi|<R\}| \leq C R^\epsilon
\]
for all $R > 1$, where here $|\cdot |$ denotes Lebesgue measure on $\RR^k$. 
\end{defn}
All measures with polynomial Fourier decay have Fourier decay outside sparse frequencies, but so do many non-Rajchman measures such as the Cantor--Lebesgue measure. 
The following result holds. 
\begin{thm}\label{thm: AFD}
A self-similar measure on $\RR^k$ has Fourier decay outside sparse frequencies if and only if it is affinely irreducible. 
\end{thm} 

There has been a significant amount of prior work showing that classes of fractal measures have decay outside sparse frequencies. 
It was verified for Bernoulli convolutions by Kaufman~\cite{K82} using an Erd\H{o}s--Kahane argument, for self-similar measures in the line (this is the $k=1$ case of Theorem~\ref{thm: AFD}) by Tsujii~\cite{Tsujii}, certain homogeneous self-similar measures in arbitrary dimensions by Mosquera and Olivo \cite[Proposition~2.2 and Section~5]{MosqueraOlivo}, and certain infinitely generated self-similar measures by Baker and Banaji \cite[Corollary~4.11]{BB25}. One could attempt to quantify the dependence between $\delta$ and $\varepsilon$ as done by Mosquera and Shmerkin in~\cite{MS18}, but that is beyond the scope of this paper. 

In Corollary~1.7 of the recent breakthrough paper~\cite{Khalil}, motivated by proving exponential mixing of geodesic flows, Khalil used an inverse entropy theorem of Hochman~\cite{HochmanRd} to prove that every measure (not necessarily self-similar) which satisfies a (local) uniform affine non-concentration condition has Fourier decay outside sparse frequencies. 
The forward implication of Theorem~\ref{thm: AFD} is straightforward. 
To establish the backward implication for (possibly overlapping) self-similar measures we introduce a weaker non-concentration condition which we call inner-affine non-concentration (see Definition~\ref{d:inner}) under which Khalil's proof of decay outside sparse frequencies also goes through, and prove that affinely irreducible self-similar measures satisfy this condition (see Theorem~\ref{thm: selfsiminner}). 

As well as its use in the proof of Theorem~\ref{thm: maingeneraldomain} under assumptions~\eqref{i:expanding} and~\eqref{i:nonexpanding}, Theorem~\ref{thm: AFD} is also used in an important way in the proof of \cite[Theorem~1.5]{BakerKhalilSahlsten}, which gives polynomial Fourier decay for certain inhomogeneous self-similar measures. 

\subsection{Further applications}
    In addition to Theorem~\ref{thm: analyticconformal} on self-conformal measures, there are several further applications of polynomial Fourier decay which can be deduced immediately by combining results from the literature with results in this paper. We briefly describe a couple of these. 
    \begin{enumerate}
    \item \emph{Normality and effective equidistribution:} We say that a point $\bx \in \RR^d$ is \emph{normal} if for every expanding integer matrix $A\in M_{d\times d}(\mathbb{Z})$, $(A^n \bx)_{n=1}^{\infty}$ is equidistributed in the torus $\RR^d / \ZZ^d$ when reduced mod $1$. 
    If the Fourier transform of a probability measure $\mu$ decays fast enough (polynomial decay suffices) then $\mu$-a.e. point is normal; see \cite[Theorem~A.1]{BakerKhalilSahlsten} (attributed to Fraser and Sahlsten). The $d=1$ case was proved earlier in~\cite{DEL,PVZZ}, and in this case one can deduce stronger quantitative equidistribution results, see for instance \cite[Theorems~1 and~3]{PVZZ}. 
    One can use \cite[Theorem~A.1]{BakerKhalilSahlsten} to deduce, for example, that if $\mu$ is Cantor--Lebesgue measure then $(x^2,x^3,x^4,\dotsc,x^{d+1})\in\mathbb{R}^d$ is normal for $\mu$-a.e. $x$. 
    \item \emph{Fourier restriction on fractals:} One can deduce Fourier restriction estimates for measures which have polynomial Fourier decay, such as those from Theorems~\ref{thm: main} and~\ref{thm: analyticconformal}. Indeed, from work such as \cite[Theorem~4.1]{Mockenhaupt}, \cite[Corollary~3.1]{Mitsis} and \cite[page~353]{Ste1993} it is known that if a measure $\mu$ on $\RR^k$ has polynomial Fourier decay then there is some $p_\mu > 1$ such that for all $p \in [1,p_\mu]$ the Fourier transform can be thought of as a bounded linear operator $L^p(\RR^k,\mbox{Lebesgue}) \to L^2(\mbox{supp}(\mu),\mu)$. 
    \end{enumerate}

\section{Proving Theorems \ref{thm: maingeneraldomain}, \ref{thm: lifttograph NE}, \ref{thm: lifttograph - expanding}, \ref{thm: polygraph} assuming Theorems \ref{thm: Lo}, \ref{thm: AFD}}\label{s:proofassumingLoandAFD}

\subsection{Preliminaries}
We first reduce the results from Section~\ref{ss:subset} to the case where the domain $U$ of the analytic maps is connected, so we can apply Theorem~\ref{thm: Lo}. 
\begin{lma}
    If Theorem~\ref{thm: maingeneraldomain}, \ref{thm: lifttograph NE}, \ref{thm: lifttograph - expanding} or~\ref{thm: polygraph} is true under the additional assumption that $U$ is connected, then it is true in general. 
\end{lma}
\begin{proof}
    Fix $\delta$ small enough that the closed $\delta$-ball around any point in the compact set $\supp(\mu)$ is contained in $U$. 
    Note that each $\delta$-branch lies in one connected component of $U$. 
    Then considering $\mu$ as a finite weighted sum of its $\delta$-branches, we see that the pushforward measure is a finite weighted sum of pushforwards of scaled copies of $\mu$, each lying in one connected component of $U$. 
    Therefore if one of the theorems holds for connected domains, then the pushforward measure is a finite sum of measures, each with the desired polynomial Fourier decay, so the corresponding theorem is true for $\mu$ itself. 
\end{proof}

 Henceforth we consider a self-similar measure $\mu$ and a smooth map $f \colon U \to \RR^d$ where $U \supset \supp(\mu)$ is open and connected. 
 Write $\Gamma_f \coloneqq \{ (\bx,f(\bx)):\bx\in U \} \subset \RR^{k+d}$ for the graph of $f$. Given $\mu_{\omega,n}\sim D_n$, let $\bx_0$ be the centre of $D_n$ and define the tangent direction
\[
T_{\omega,n}\coloneqq\{(\mathbf{h},Df(\bx_0)\mathbf{h}):\mathbf{h}\in\RR^k\}.
\]
For sufficiently large $n$, the cubes under consideration lie in $U$. Given $\bxi\in\RR^{k+d}$ with first $k$ coordinates $0$, let $\bxi_{\omega,n}\in\RR^k$ be the projection of $\bxi$ to $\RR^k\times\bzero$ along the orthogonal complement of $T_{\omega,n}$, i.e.
\begin{equation}\label{e:definexiomegan}
((\bxi_{\omega,n},\bzero)-\bxi) \perp T_{\omega,n}.
\end{equation}
We first use the following result from our previous paper. 

\begin{lma}[Lemmas~3.2 and~3.3 from \cite{BY-quantitative}]\label{lma: startingdecomp}
    Let $\mu$ be a self-similar measure on $\RR^k$, let $U$ be an open neighbourhood of $\supp(\mu)$, let $f \colon U \to \RR^d$ be $C^2$, and let $T(\bx) = (\bx,f(\bx))$. 
    Then for sufficiently large $n \in \NN$ and $\bxi \in \RR^{k+d}$ with first $k$ coordinates being $0$, 
    \[ |\widehat{\mu_f}(\xi_{k+1},\dotsc,\xi_{k+d})| = |\widehat{\mu_T}(\bxi)| \leq \sum_{D_n}\sum_{\mu_{\omega,n}\sim D_n} \left|\widehat{\mu}_{\omega,n}(\bxi_{\omega,n})\right|+O_{\mu,f}(|\bxi|/2^{2n}). \]
\end{lma}
\begin{proof}
    Choose a $C^2$ extension $F\colon\RR^k\to\RR^d$ which agrees with $f$ on a neighbourhood of $\supp(\mu)$, using a smooth cutoff supported inside $U$. For sufficiently large $n$, every $3D_n$ under consideration lies in this neighbourhood, so the tangent directions and the projected frequencies are the same for $F$ and $f$. Applying \cite[Lemmas~3.2 and~3.3]{BY-quantitative} to $F$ gives the result.
\end{proof}

\subsection{Non-expanding self-similar measures}
We first show Theorem~\ref{thm: maingeneraldomain} under assumption~\eqref{i:nonexpanding}. 
Let $U$ be an open connected set containing the support of $\mu$ on which $f$ is real analytic.

For each $\bv\in\mathbb{S}^{d-1}$ we defined $f_\bv$ to be the function given by the inner product $\bx \mapsto \langle f(\bx),\bv \rangle$. 
Next, for $\xi>0$, we define $P_{\xi\bv}\coloneqq\xi P_\bv$, consistent with~\eqref{e:definepv}. 
We see that $P_\bv$ is real analytic. 
For each $\by\in\mathbb{R}^k$, consider the set
\begin{equation}\label{e:EvyDef}
E_{\bv,\by} \coloneqq P^{-1}_\bv(\{\by \}).
\end{equation}
We see that $E_{\bv,\by}$ is a (not necessarily proper/non-singular) analytic subvariety of $U$ defined by the analytic equation $\{\bx\in U:P_\bv(\bx)-\by = \bzero\}$. 
The correspondence between varieties and equations is fixed in this way. 
We denote this collection of subvarieties by $\cM_f$. 

If for some $\bv,\by$ the subvariety $E_{\bv,\by}$ has dimension $k$, then $E_{\bv,\by}$ contains an open ball, and $P_\bv$ is a constant map on some ball, so $f$ is degenerate. 
Next, we consider the case that $f$ is non-degenerate, so the $E_{\bv,\by}$ never have dimension $k$, and are always proper analytic subvarieties of $U$. 
\begin{lma}\label{lem: gen NE twisters}
    Let $\mu$ be a non-expanding self-similar measure on $\mathbb{R}^k$ with Fourier decay outside sparse frequencies, and let $U$ be an open neighbourhood of $\supp(\mu)$. Let $f\colon U\to\mathbb{R}^d$ be real analytic and non-degenerate. 
    Suppose $\mu$ uniformly decays near $\cM_f$. 
    Then $\mu_f$ has polynomial Fourier decay.
\end{lma}

Here, `$\mu$ uniformly decays near $\cM_f$' means that for each $\varepsilon>0$ there exist $c,\eta > 0$ (depending on $\varepsilon$) such that for all $\bv\in\mathbb{S}^{d-1}$, $\by\in\mathbb{R}^k$ and $\delta>0$,
\[
\mu(\{\bx\in U:|P_\bv(\bx)-\by|\leq \delta\}^{(\delta^\varepsilon)}) \leq c \delta^\eta,
\]
where $P_\bv(\cdot)-\by$ is the defining analytic function for $E_{\bv,\by}$ as mentioned above. 
\begin{proof}[Proof of Lemma~\ref{lem: gen NE twisters}]
Let $n$ be a large integer and let $\bxi \in \RR^{k+d}$ with first $k$ coordinates being $0$ and $|\bxi|\asymp 2^{1.5n}$.\footnote{If one were trying to quantify the exponent, one could take $|\bxi|\asymp 2^{\gamma n}$ for a more carefully chosen $\gamma \in (1,2)$, as in~\cite{BY-quantitative}.} 
We start with Lemma~\ref{lma: startingdecomp} which gives 
\begin{align}\label{eqn: decom}
|\widehat{\mu_T}(\bxi)| \leq \sum_{D_n}\sum_{\mu_{\omega,n}\sim D_n} \left|\widehat{\mu}_{\omega,n}(\bxi_{\omega,n})\right|+O(|\bxi|/2^{2n}).
\end{align}
For each scaled and rotated copy $\mu_{\omega,n}$ of $\mu$, write the defining similarity as $\bx\mapsto r_{\omega,n}O_{\omega,n}^T\bx+\bt_{\omega,n}$, where $r_{\omega,n}\asymp 2^{-n}$ and $O_{\omega,n}\in O_k(\RR)$. Then
\[
|\widehat{\mu}_{\omega,n}(\bxi_{\omega,n})|=|\mu_{\omega,n}|\,|\widehat{\mu}(r_{\omega,n}O_{\omega,n}\bxi_{\omega,n})|.
\]
Thus $O_{\omega,n}$ denotes the transpose of the spatial rotation. 
It is convenient to define $L_n$ to be the following set of probability weights, scaling ratios and rotation s of $\mu_{\omega,n}$: 
\begin{align*}
L_n=&\left\{(p,r,R): \text{$p,r,R$ are the probability weight, scaling ratio} \right. \\ 
&\left.\text{and rotation for some $\mu_{\omega,n}$} \right\}.
\end{align*}
Thus, $L_n$ offers a way to classify branches $\mu_{\omega,n}$. We rewrite~\eqref{eqn: decom} according to this classification: 
\[
|\widehat{\mu_T}(\bxi)| \leq \sum_{g\in L_n}\sum_{D_n\in\mathcal{D}_n}\sum_{\substack{\mu_{\omega,n}\sim g \\ \mu_{\omega,n}\sim D_n}} |\mu_{\omega,n}| |\widehat{\mu}(r_{\omega,n} O_{\omega,n}(\bxi_{\omega,n}))|+O(|\bxi|/2^{2n}).\]
For $\mu_{\omega,n}\sim g$, we write $2^{-\kappa_g n}=|\mu_{\omega,n}|$ for the probability weight of those $\mu_{\omega,n}$. 
We now define $\mathcal{C}_n$ to be the collection of all $D_0$ (dyadic cubes of unit scale) that intersect $\{r_{\omega,n}O_{\omega,n}(\bxi_{\omega,n}):\omega\in\mathcal{W}_n\}$. 
Note that cubes in $\mathcal{C}_n$ lie distance $\ll |\bxi/2^{n}|$ from the origin. 
For each $g\in L_n$, we write $\mu_{\omega,n} \sim g$ if it has the probability weight, scaling ratio and rotation indicated by $g$. 
Notice that $\#L_n\ll 2^{\epsilon n}$ for each $\epsilon>0$ because of the non-expanding property. Therefore 
\begin{align*}
|\widehat{\mu_T}(\bxi)| &\ll \sum_{g\in L_n} \sum_{D_0\in\mathcal{C}_n}  \sum_{\substack{\mu_{\omega,n}\sim g \\ r_{\omega,n}O_{\omega,n}(\bxi_{\omega,n})\in D_0}} |\mu_{\omega,n}| |\widehat{\mu}(r_{\omega,n}O_{\omega,n}(\bxi_{\omega,n}))| + |\bxi|/2^{2n} \\
&\ll \sum_{g\in L_n}\frac{1}{2^{\kappa_g n}}\sum_{D_0\in\mathcal{C}_n}\sum_{\substack{\mu_{\omega,n}\sim g \\ r_{\omega,n}O_{\omega,n}(\bxi_{\omega,n}) \in D_0}} |\widehat{\mu}(r_{\omega,n} O_{\omega,n}(\bxi_{\omega,n}))| + |\bxi|/2^{2n}.
\end{align*}
Note that the rotation part $O_{\omega,n}$ of $\mu_{\omega,n}$ depends only on $g$. From here we write
\[
N_{D_0}(g) \coloneqq \#\{\mu_{\omega,n}:\mu_{\omega,n}\sim g, r_{\omega,n}O_{\omega,n}(\bxi_{\omega,n}) \in  D_0\}.
\]
Then we see that
\begin{equation}\label{eqn:upperboundqual}
|\widehat{\mu_T}(\bxi)|\ll  \sum_{g\in L_n} \frac{1}{2^{\kappa_g n}}\sum_{D_0\in\mathcal{C}_n} N_{D_0}(g) \sup_{\bxi' \in D_0} |\widehat{\mu}(\bxi')| + |\bxi|/2^{2n}.
\end{equation}

Since $\mu$ has Fourier decay outside sparse frequencies, the unit-cube formulation justified in Section~\ref{s:afd} shows that for each $\epsilon>0$ we can find a $\delta>0$ so that 
\begin{equation}\label{e:dosfconsequence}
    \# \{ D_0 \in \mathcal{C}_n : \exists \bxi'\in D_0 \mbox{ s.t. } |\widehat{\mu}(\bxi')|\geq |\bxi/2^n|^{-\delta} \} \ll 2^{\varepsilon n}.
\end{equation}
We then have
\begin{align*}
& \sum_{g\in L_n}\frac{1}{2^{\kappa_g n}} \sum_{D_0\in\mathcal{C}_n} N_{D_0}(g)  \sup_{\bxi' \in D_0} |\widehat{\mu}(\bxi')|  \\
\ll &  \sum_{g\in L_n}\frac{1}{2^{\kappa_g n}} \sum_{D_0\in\mathcal{C}_n} N_{D_0}(g) |\bxi/2^n|^{-\delta}\\
&+ \sum_{g\in L_n}\frac{1}{2^{\kappa_g n}} \sum_{\substack{D_0: \exists\bxi'\in D_0, \\ |\widehat{\mu}(\bxi')|\geq |\bxi/2^n|^{-\delta}}} N_{D_0}(g) + |\bxi|/2^{2n}. 
\end{align*}
Observe that because $\mu$ is a probability measure,
\begin{equation*}
 \sum_{g\in L_n}\frac{1}{2^{\kappa_g n}} \sum_{D_0\in\mathcal{C}_n} N_{D_0}(g) \ll\sum_{g\in L_n}\frac{1}{2^{\kappa_g n}} \#\{\mu_{\omega,n}:\mu_{\omega,n}\sim g\} \ll 1.
\end{equation*}
On the other hand, because of~\eqref{e:dosfconsequence}, 
\[
 \sum_{g\in L_n} \frac{1}{2^{\kappa_g n}}\sum_{\substack{D_0: \exists\bxi'\in D_0, \\ |\widehat{\mu}(\bxi')|\geq |\bxi/2^n|^{-\delta}}} N_{D_0}(g)\ll \sum_{g\in L_n} \frac{1}{2^{\kappa_g n}} 2^{\epsilon n}   \max_{D_0} N_{D_0}(g).
\]

Next, we need the following claim, whose proof will use the decay property for $\mu$ near $\cM_f$. 
\begin{claim}\label{claim:boundnumber}
    There exists $\eta>0$ such that, uniformly for all $D_0$ under consideration,
\begin{equation}\label{eqn: preimage}
N_{D_0}(g) = \#\{\mu_{\omega,n}:\mu_{\omega,n}\sim g, \bxi_{\omega,n}\in r^{-1}_gR^{-1}_g(D_0)\} \ll \left(\frac{2^n}{|\bxi|}\right)^{\eta} 2^{\kappa_g n}.
\end{equation}
Here $r_g$ is the scaling ratio corresponding to $g$ and $R_g$ is the rotation corresponding to $g$. 
\end{claim}
From this claim, we obtain
\[
 \sum_{g\in L_n} \frac{1}{2^{\kappa_g n}}\sum_{\substack{D_0: \exists\bxi'\in D_0, \\ |\widehat{\mu}(\bxi')|\geq |\bxi/2^n|^{-\delta}}}  N_{D_0}(g)\ll 2^{2\epsilon n}  \left(\frac{2^n}{|\bxi|}\right)^{\eta}.
\]
Given the claim, \eqref{eqn:upperboundqual} gives 
\begin{align}\label{eq: choosing xi}
|\widehat{\mu_T}(\bxi)|\ll |\bxi/2^n|^{-\delta}+2^{2\epsilon n} (2^n/|\bxi|)^{\eta} + |\bxi|/2^{2n}.
\end{align}
Since we chose $|\bxi|\asymp 2^{1.5n}$, and we can choose $\varepsilon$ to be small enough depending on $\eta$, this implies that
\[
|\widehat{\mu_T}(\bxi)|\ll |\bxi|^{-\sigma}
\]
for some $\sigma>0$, as required. Since the above holds for all $n$ and all choices of $\bxi$ with $|\bxi|\asymp 2^{1.5n}$, this finishes the proof. 
\end{proof}
\begin{proof}[Proof of Claim~\ref{claim:boundnumber}]
We need the decay property for $\mu$ near $\cM_f$, which we assumed to hold. Write $\bxi=(\bzero,\xi\bv)$, where $\bv\in\mathbb{S}^{d-1}$ and $\xi=|\bxi|$. Recalling~\eqref{e:definepv}, we have
\[
P_{\xi\bv}\colon U\to\RR^k,\qquad P_{\xi\bv}(\bx)=Df(\bx)^T(\xi\bv)=\xi P_\bv(\bx).
\]
By~\eqref{e:definexiomegan}, $\bxi_{\omega,n}=P_{\xi\bv}(\bx_0)$, where $\bx_0$ is the centre of the cube associated with $\mu_{\omega,n}$. Thus the centres corresponding to branches counted by $N_{D_0}(g)$ lie in $P_\bv^{-1}(D'')$, where
\[
D''\coloneqq (r_g\xi)^{-1}R_g^{-1}(D_0).
\]
This is a rotated cube of diameter $\asymp 2^n/\xi$. Let $\by$ be its centre and put $\delta=C2^n/\xi$ for a sufficiently large fixed $C$. Then each of the corresponding centres lies in $\{\bx\in U:|P_\bv(\bx)-\by|<\delta\}$. The support of the associated branch lies within distance $C'2^{-n}$ of that centre.

Choose any fixed $\varepsilon\in(0,2)$. Since $\xi\asymp 2^{1.5n}$, we have $C'2^{-n}<\delta^\varepsilon$ for sufficiently large $n$. Consequently, every branch counted by $N_{D_0}(g)$ is supported in the $\delta^\varepsilon$-neighbourhood of this sublevel set. 
Therefore 
\[
2^{-\kappa_g n}N_{D_0}(g)\leq\mu(\{\bx\in U:|P_\bv(\bx)-\by|<\delta\}^{(\delta^\varepsilon)})\ll\delta^\eta.
\]
This proves the claim, whether or not the branch supports overlap.
\end{proof}

\begin{proof}[Proof of Theorem~\ref{thm: maingeneraldomain}~\eqref{i:nonexpanding}]
All that remains is to verify the hypotheses of Lemma~\ref{lem: gen NE twisters}. 
Decay outside sparse frequencies follows from the affine irreducibility of $\mu$ and Theorem~\ref{thm: AFD}. 
It remains to verify the uniform decay property. 
By Theorem~\ref{thm: Lo}, we need to verify that the collection of analytic functions $P_\bv(\cdot )-\by$ is a compact family excluding the zero function. 
Choose a compact neighbourhood $K'\subset U$ of $K=\supp(\mu)$, and let $M$ bound $|P_\bv|$ on $K'$ uniformly in $\bv$. Choose $\delta_0\in(0,1)$ such that $\delta_0<1/2$ and $\delta_0^\varepsilon<\operatorname{dist}(K,\RR^k\setminus K')$. If $|\by|>M+1$ and $0<\delta<\delta_0$, the thickened sublevel set has zero $\mu$-measure. For $\delta\geq\delta_0$, the desired estimate follows by increasing its constant, uniformly in $\by$, since $\mu$ is a probability measure. Therefore it suffices to consider $|\by|\leq M+1$, so that $\bv,\by$ range over compact sets. Observe that
\[
P_\bv(\bx)-\by=\sum_{i=1}^d v_i \nabla f_i(\bx)-\by.
\]
Each $f_i$ is analytic, hence $\nabla f_i$ is analytic as well. 
Denote the radius of convergence of $\nabla f_i$ at $\bx$ by $r_i(\bx)>0$. Then the convergence radius of $P_\bv(\cdot )-\by$ at $\bx$ is at least
\[
\min_i\{r_i(\bx): v_i\neq 0\}.
\]
For each fixed $i$ and each compact $K\subset U$, $\inf_{\bx\in K}r_i(\bx)>0$. To see this, for each $\bx$, consider $\bx'\in B_{r_i(\bx)/2}(\bx)$. Then $r_i(\bx')\geq r_i(\bx)/4$. Since finitely many $B_{r_i(\bx)/4}(\bx)$ cover $K$, we conclude that 
\[
\inf_{\bx\in K}r_i(\bx)\geq \min_\bx\{r_i(\bx)/4\}>0
\]
where in the middle, $\bx$ ranges over a finite set, not over $K$. Hence, each $\nabla f_i$, as well as $P_\bv(\cdot )-\by$, can be treated as complex analytic in some complex neighbourhood of $K\subset U$. The complex-compact-open compactness follows as $\bv,\by$ all range over compact subsets. 
The strong compactness of the family follows. 
The non-containment of the zero function follows from the non-degeneracy of $f$. 
\end{proof}

\subsection{Expanding self-similar measures}
The proof of Theorem~\ref{thm: maingeneraldomain} under the assumption~\eqref{i:expanding} is similar to the proof under assumption~\eqref{i:nonexpanding}. 
If $|P_\bv|$ is not constant, then for each $r\in [0,\infty)$, the set \[
E_{\bv,r}=\{|P_\bv(\bx)|=r\}
\]
is an analytic variety. Its dimension cannot be $k$, otherwise $|P_\bv|^2$ (which is analytic) would be a constant function. Those subvarieties form a collection $\cM^{K}_f$ ($K$ for ``cone''). For such an $E_{\bv,r}$, we assign the defining equation to be
\[
|P_\bv(\bx)|^2=r^2.
\]
This is because $\bx \mapsto |P_\bv(\bx)|^2-r^2$ is real analytic. We have the following result.
\begin{lma}\label{lem: gen expanding twisters}
    Let $\mu$ be a self-similar measure on $\RR^k$ with Fourier decay outside sparse frequencies and let $U$ be an open connected neighbourhood of $\supp(\mu)$. Let $f\colon U \to\mathbb{R}^d$ be real analytic and assume that for every $\bv \in \mathbb{S}^{d-1}$, $|P_\bv|$ is non-constant. 
    Suppose $\mu$ uniformly decays near $\cM^K_f$. Then $\mu_f$ has polynomial Fourier decay.
\end{lma}
\begin{proof}
    The proof is very similar to that of Lemma~\ref{lem: gen NE twisters}. We need to rectify the difficulty that our self-similar system may now have a large rotation group. 

    We can follow the proof of Lemma~\ref{lem: gen NE twisters} until the construction of $L_n$. 
    Because the rotation group may not be non-expanding we must not include it in $L_n$. The new $L_n$ is defined as
    \begin{align*}
L_n=\{(p,r):&\ \text{$p,r$ are the probability weight and scaling ratio}\\
&\ \text{for some $\mu_{\omega,n}$}\}.
\end{align*}
Notice that $\#L_n$ grows only polynomially with respect to $n$. 
We can then follow the rest of the proof of Lemma~\ref{lem: gen NE twisters} until we hit the construction of $N_{D_0}(g)$. It is now
\[
N_{D_0}(g) \coloneqq \#\{\mu_{\omega,n}:\mu_{\omega,n}\sim g, r_{\omega,n} O_{\omega,n}(\bxi_{\omega,n})\in  D_0\}.
\]
We can no longer control the rotation part $O_{\omega,n}$. Nonetheless, in analogy with Claim~\ref{claim:boundnumber} we claim that
\begin{equation*}
N_{D_0}(g) \ll \#\{\mu_{\omega,n}:\mu_{\omega,n}\sim g, \bxi_{\omega,n}\in A_{D_0/r_g}\} \ll \left(\frac{2^n}{|\bxi|}\right)^{\eta} 2^{\kappa_g n}.
\end{equation*}
Here $r_g$ is the scaling ratio corresponding to $g$ and $A_{D_0/r_g}$ is the set of points $\bxi$ in $\mathbb{R}^k$ so that $R(\bxi)\in D_0/r_g$ for some rotation $R$. Thus $A_{D_0/r_g}$ is an annulus of thickness $\asymp 1/r_g\asymp 2^n$. This annulus can have an inner radius $0$.

Write $\bxi=(\bzero,\xi\bv)$ as in Claim~\ref{claim:boundnumber}. After division by $\xi$, the annulus above has thickness $\ll 2^n/\xi$. Put $\delta=2^n/\xi$. If no corresponding cube centre has its $P_\bv$-image in this annulus, there are no branches to count. Otherwise, since $|P_\bv|$ is uniformly bounded near $\supp(\mu)$, we can choose a uniformly bounded radius $a\geq 0$ such that every corresponding centre $\bx_0$ satisfies
\[
\big||P_\bv(\bx_0)|-a\big|\ll\delta,\qquad
\big||P_\bv(\bx_0)|^2-a^2\big|\ll\delta.
\]
The second inequality follows by multiplying the first by $|P_\bv(\bx_0)|+a$, which is uniformly bounded. It also holds when the inner radius of the annulus is $0$.
As in Claim~\ref{claim:boundnumber}, each counted branch is supported within distance $C2^{-n}$ of a centre in this sublevel set. Taking a fixed $\varepsilon\in(0,2)$, the uniform decay near $\cM_f^K$ gives
\[
2^{-\kappa_g n}N_{D_0}(g)\ll\delta^\eta.
\]
The remainder of the proof follows as in Lemma~\ref{lem: gen NE twisters}.
\end{proof}

\begin{proof}[Proof of Theorem~\ref{thm: maingeneraldomain}~\eqref{i:expanding}]
We apply Lemma~\ref{lem: gen expanding twisters}. 
Decay outside sparse frequencies follows from Theorem~\ref{thm: AFD}. 
Choose a compact neighbourhood $K'\subset U$ of $\supp(\mu)$, and let $M$ bound $|P_\bv|$ on $K'$ uniformly in $\bv$. For $r>M+1$, the function $|P_\bv|^2-r^2$ is bounded away from zero on $K'$. Thus, for sufficiently small $\delta$, the thickened sublevel set has zero $\mu$-measure, uniformly in $r$; larger $\delta$ are handled by increasing the constant, as in the non-expanding case. We may therefore restrict to $0\leq r\leq M+1$ and obtain a compact family $\{|P_\bv(\cdot )|^2-r^2\}_{\bv,r}$. 
Here, $|P_\bv(\cdot )|^2$ is the square sum of components of $P_\bv(\cdot )$, which is real analytic. Its convergence radius is related in a similar way to the convergence radius of the components as in the proof of Theorem \ref{thm: maingeneraldomain} above. We extend the square sum (rather than the norm square) to the complex domain. 
The complex-compact-open compactness follows similarly. This family does not contain the zero function since $f$ is non-conical. 
Using Theorem~\ref{thm: Lo} then gives the decay property and finishes the proof. 
\end{proof}

\subsection{Self-similar measures with polynomial Fourier decay}
We can also use similar methods to prove Theorem~\ref{thm: maingeneraldomain} under assumption~\eqref{i:selfsimhaspoly}. 
As discussed above, we assume $U$ is connected. 
\begin{proof}[Proof of Theorem \ref{thm: maingeneraldomain} \eqref{i:selfsimhaspoly}]
Our non-trapped assumption on $f$ means that for $\bv \in \mathbb{S}^{d-1}$, $\bx \mapsto f_{\bv}(\bx)$ can be affine but can never be constant on any ball intersecting $\supp(\mu)$. 
In other words, $\nabla f_\bv$ can be a constant function on $U$, but it is not identically zero on any component. 
We define $E_{\bv,\bzero}$ as in~\eqref{e:EvyDef}. We see that for all $\bv \in \mathbb{S}^{d-1}$, $E_{\bv,\bzero}$ is a proper analytic subvariety with dimension $<k$. 
The family $\{P_\bv:\bv\in\mathbb{S}^{d-1}\}$ is compact and does not contain the zero function. Thus Theorem~\ref{thm: Lo} shows that $\mu$ uniformly decays near the corresponding zero sets $E_{\bv,\bzero}$, i.e. for all $\varepsilon > 0$ there exists $\eta > 0$ such that 
\begin{equation}\label{e:polypushdeltadelta}
\mu( \{\bx:|P_\bv(\bx)|<\delta\}^{(\delta^{\varepsilon})}) \ll \delta^{\eta} 
\end{equation}
uniformly for $\delta > 0$ and $\bv \in \mathbb{S}^{d-1}$. 

By assumption, there exists $\sigma>0$ such that $|\widehat{\mu}(\bxi')| \ll |\bxi'|^{-\sigma}$ for $\bxi' \in \RR^k \setminus \{0\}$. 
Fix $\gamma \in (1,2)$. Since we are not concerned with optimising the exponent here, the precise choice of $\gamma$ does not matter, so we can choose $\gamma = 1.5$ for concreteness. 
Fix $0<\tau<1$ small enough that $\gamma - 1 - \tau > 0$. 
Let $\bv \in \mathbb{S}^{d-1}$ (we will prove uniform polynomial Fourier decay estimates over such $\bv$), and let $\bxi = \xi \bv$ for some $\xi \in [2^{\gamma n},2^{\gamma (n+1)})$. 
For large $n \in \NN$, we divide the set of dyadic cubes in $\mathcal{D}_n$ which intersect $\supp(\mu)$ into two sets, called $\mathcal{C}_n$ and $\mathcal{C}_n'$, where $\mathcal{C}_n$ consists of those which intersect the $2^{-\tau n}$-neighbourhood of $\{\bx:|P_\bv(\bx)|<2^{-\tau n}\}$ and $\mathcal{C}_n'$ consists of those which do not. 
We may assume $n$ is sufficiently large that all cubes in $\mathcal{C}_n \cup \mathcal{C}_n'$ are contained in $U$, the domain of $f$. Let $\mu_{\omega,n}$ be associated with a cube of level $n$. For the graph frequency $(\bzero,\bxi)\in\RR^{k+d}$, consider its projected frequency $\bxi_{\omega,n}\in\RR^k$, $r_{\omega,n}\asymp 2^{-n}$ (the scaling ratio of $\mu_{\omega,n}$) and $O_{\omega,n}\in O_k(\RR)$. Recalling the relationship between $\bxi_{\omega',n}$ and $P_{\bv}$, we see that the magnitude of $\bxi_{\omega',n}$ associated with a cube in $\mathcal{C}_n'$ is $\gg 2^{(\gamma - \tau)n}$. 
Starting now from Lemma~\ref{lma: startingdecomp} and using~\eqref{e:polypushdeltadelta}, 
\begin{align*}
     |\widehat{\mu_f}(\bxi)| &\leq \sum_{D_n \in \mathcal{C}_n}\sum_{\mu_{\omega,n}\sim D_n} \left|\widehat{\mu}_{\omega,n}(\bxi_{\omega,n})\right| \\
     &\quad + \sum_{D_n' \in \mathcal{C}_n'}\sum_{\mu_{\omega',n}\sim D_n'} \left|\widehat{\mu}_{\omega',n}(\bxi_{\omega',n})\right| + O(|\bxi|/2^{2n}) \\
     &\ll 2^{-\tau \eta n} +  \sum_{D_n' \in\mathcal{C}_n'} \sum_{\mu_{\omega',n}\sim D_n'} |\mu_{\omega',n}| |\widehat{\mu}(r_{\omega',n} O_{\omega',n}(\bxi_{\omega',n}))|   + |\bxi|^{1 - 2/\gamma} \\
     &\ll |\bxi|^{-\tau \eta / \gamma}   +  2^{-(\gamma - \tau - 1) \sigma n}   +   |\bxi|^{1 - 2/\gamma} \\
     &\ll |\bxi|^{-\min\{ \tau \eta / \gamma , (\gamma - \tau - 1)\sigma / \gamma , 2/\gamma - 1 \}} ,
\end{align*}
as required. 
\end{proof}

\subsection{Lift to the graph}
Next, we prove the results concerning lifts to the graph. 
Instead of applying Theorem~\ref{thm: Lo}, we will need to use a more quantitative version (Proposition~\ref{p:quanifiedlo}), as well as a general fact about compact families of analytic functions (Lemma~\ref{lma: choice of D}). 

\begin{proof}[Proof of Theorem \ref{thm: lifttograph NE}] 
Consider the map $T\colon \bx \mapsto (\bx,f(\bx))$. We consider the function $P^T_\bv$ for some $\bv \coloneqq (\bv_k,\bv_d)\in \mathbb{S}^{k+d-1}$ (with $\bv_k,\bv_d$ the corresponding components in the first $k$ and last $d$ coordinates in $\mathbb{R}^{k+d}$), defined as
\begin{equation}\label{e:defineptv}
P^T_\bv(\bx) \coloneqq \bv_k+Df(\bx)^T\bv_d.
\end{equation}
This is well-defined for all $\bv\in\mathbb{S}^{k+d-1}$. If $\bv_d\neq\mathbf{0}$, as assumed in this theorem, then $P^T_\bv=\bv_k+|\bv_d|P_{\bv_d/|\bv_d|}$.

Next, consider the family
\[
\mathcal{G}\coloneqq\{P_{\bu}(\cdot)-P_{\bu}(\bz):\bu\in\mathbb{S}^{d-1},\ \bz\in\supp(\mu)\}.
\]
This is a compact family in the sense of Definition~\ref{d:compact}, since the parameters range over compact sets and the functions have common holomorphic extensions near each compact subset of $U$. By the non-degeneracy assumption, $\mathcal{G}$ does not contain the zero function.
Applying Lemma~\ref{lma: choice of D} to $\mathcal{G}$ and $K=\supp(\mu)$, we obtain an integer $D>0$ and a positive number $C$ such that, for every $\bu\in\mathbb{S}^{d-1}$ and $\bz\in\supp(\mu)$, at least one Taylor coefficient of $P_{\bu}(\cdot)-P_{\bu}(\bz)$ at $\bz$ with degree at most $D$ has absolute value larger than $C$. The constant coefficient is zero, so this coefficient has degree in $\{1,\dotsc,D\}$ and is also a coefficient of $P_{\bu}$ at $\bz$.
From~\eqref{e:defineptv}, for every $\bv=(\bv_k,\bv_d)\in\mathbb{S}^{k+d-1}$ with $\bv_d\neq\mathbf{0}$, at least one positive-degree coefficient of $P^T_{\bv}$ at each $\bz\in\supp(\mu)$ has absolute value larger than $C|\bv_d|$ and degree at most $D$. For each fixed such $\bv$, the same lower bound holds for $P^T_{\bv}-\by$ as $\by$ ranges over a compact set. We can therefore apply Proposition~\ref{p:quanifiedlo} to these compact families, just as we used Theorem~\ref{thm: Lo} to prove Theorem~\ref{thm: maingeneraldomain}~\eqref{i:nonexpanding}. Crucially, $D$ is independent of $\bv$, and the exponent in Proposition~\ref{p:quanifiedlo} is independent of the coefficient lower bound $C|\bv_d|$. Thus the Fourier decay exponent can be chosen independently of $\bv$, while the multiplicative constant may depend on $\bv$.
\end{proof}

\begin{proof}[Proof of Theorem~\ref{thm: lifttograph - expanding}]
Recall the definition of $P^T_{\bv}$ from~\eqref{e:defineptv}. 
If $|P^T_\bv|$ were a constant function of $\bx$, then we see that the function defined by
\[
\tilde{f}\colon \bx\mapsto |\bv_d|f(\bx)+\langle \bv_k,\bx \rangle \bv_d/|\bv_d|
\]
would be conical. 
This would mean that $f$ is the sum of a conical function and an affine function, contradicting the hypothesis. 
So we have shown that $|P^T_\bv|$ is not constant when $\bv_d\neq\mathbf{0}$. As in the expanding case above, a uniform bound for $|P^T_\bv|$ on a compact neighbourhood of $\supp(\mu)$ allows us to restrict $r$ to a compact interval. Fixing $\varepsilon>0$ and restricting to $|\bv_d|\geq\varepsilon$, we therefore obtain a compact family $\cF_\varepsilon=\{|P^T_\bv(\cdot)|^2-r^2\}_{\bv,r}$ which does not contain the zero function. 
Therefore as in the proof of Theorem~\ref{thm: maingeneraldomain}~\eqref{i:expanding}, we have polynomial Fourier decay with uniform exponent and constant over directions $(\bv_k,\bv_d)\in \mathbb{S}^{k+d-1}$ with $|\bv_d| \geq \varepsilon$. 
\end{proof}

\begin{proof}[Proof of Theorem \ref{thm: polygraph}]
    Let $\bv \coloneqq (\bv_k,\bv_d)\in \mathbb{S}^{k+d-1}$. 
    Define $P_{\bv}^T$ as in~\eqref{e:defineptv}.
    If $\bv_d = \mathbf{0}$ then $P_{\bv}^T=\bv_k$ is constant and non-zero. 
    If $\bv_d \neq \mathbf{0}$ then $P_{\bv}^T$ is non-constant by the non-degeneracy assumption. 
    Therefore $\{P_{\bv}^T\}_{\bv\in\mathbb{S}^{k+d-1}}$ forms a compact family of analytic functions that does not contain the zero function. 
    Thus as in the proof of Theorem~\ref{thm: maingeneraldomain}~\eqref{i:selfsimhaspoly} (using the polynomial Fourier decay of $\mu$), we get the desired polynomial Fourier decay for $\mu_T$. 
\end{proof}

\section{Uniform measure \L ojasiewicz inequality: Theorem \ref{thm: Lo}}\label{s:lo}

\subsection{Self-similar measures of proper varieties}

We first record the following basic fact, which will also be used in Section~\ref{s:afd}. 

\begin{lma}\label{l:tube}
    Let $\mu$ be a finite Borel measure on $\RR^k$ whose support is compact and does not lie in any affine hyperplane. 
    Then there exist $\delta,\gamma \in (0,1)$ such that for all affine hyperplanes $K$ of $\RR^k$, 
    \[ \mu(K^{(\gamma)}) \leq (1-\delta) \mu(\RR^k). \]
\end{lma}
\begin{proof}
    Assume for contradiction this is false.
    Then there exists a sequence $(S_n)_{n=1}^{\infty}$ of affine hyperplanes in $\RR^k$ such that $\mu(S_n^{(1/n)}) \geq \mu(\RR^k) - 1/n$.
    Let $B$ be a ball containing the support of $\mu$. For all sufficiently large $n$, $S_n^{(1/n)}$ meets $B$. Write
    \[ S_n = \{\bx\in\RR^k:\langle\bv_n,\bx\rangle=a_n\} \]
    for some $\bv_n \in \RR^k$ with $|\bv_n| = 1$ and $a_n \in \RR$. 
    The numbers $a_n$ are bounded, so by compactness we can pass to a subsequence such that $\bv_n\to\bv$ and $a_n\to a$, where $|\bv|=1$. 
    Let $S \coloneqq \{\bx:\langle\bv,\bx\rangle=a\}$.
    For each $\varepsilon>0$ and all sufficiently large $n$, we have $S_n^{(1/n)}\cap B\subseteq S^{(\varepsilon)}$. Hence $\mu(S^{(\varepsilon)})=\mu(\RR^k)$ for every $\varepsilon>0$.
    By continuity from above, $\mu(S)=\mu(\RR^k)$, contradicting our assumption that the support of $\mu$ does not lie in an affine hyperplane.
\end{proof}

We need the following lemma, which shows that affinely irreducible self-similar measures cannot `see' proper analytic varieties. 
More generally, it proves the same for Cartesian products of self-similar measures, which may not themselves be self-similar. 

\begin{lma}\label{lma: prod of SS}
    Let $\mu$ be a Cartesian product of (possibly different) self-similar measures. 
    Suppose that $\mu$ is supported inside a convex open set $U \subset \RR^k$ and let $f_1,\dotsc,f_n \colon U \to \RR$ be real analytic functions. 
    Consider the analytic subvariety 
    \[ M \coloneqq \{ \bx \in U : f_1(\bx) = \dotsb = f_n(\bx) = 0 \} \] 
    and assume that $\mu(M)>0$. 
    Then there is a (possibly full-dimensional) affine subspace $L$ such that $L \cap U \subseteq M$ and $\mu(L) = 1$. 
\end{lma}

\begin{proof}
    Write $\mu=\mu_1\times\dotsb\times\mu_j$, where $\mu_i$ is a self-similar measure on $\RR^{k_i}$ and $k_1+\dotsb+k_j=k$. We first assume that $\supp(\mu)$ is not contained in an affine hyperplane. By rescaling the factors, we may assume that each $\supp(\mu_i)$ has diameter at most $1$.
    Let $r_{\min}>0$ be the smallest contraction ratio of all the IFS maps generating the factors. At scale $\delta\in(0,1)$, stop each word in the $i$th IFS as soon as its contraction ratio is at most $\delta$. The resulting words have contraction ratios in $[r_{\min}\delta,\delta]$ and their cylinder sets partition the corresponding symbolic space.
    Taking Cartesian products, we obtain normalised $\delta$-product branches of the form $\mu_{s_B}$ with 
    \[ s_B(\bx) = \bt+\delta A\bx,\qquad
    A = \operatorname{diag}(c_1O_1,\dotsc,c_jO_j),
    \]
    where $c_i\in[r_{\min},1]$ and $O_i\in O_{k_i}(\RR)$. 
    Let $\mathcal{A}$ be the compact family of all these block diagonal matrices. 
    Every $A\in\mathcal{A}$ is invertible and has smallest singular value at least $r_{\min}$.

    \begin{claim}\label{c:distortedproducts}
        There exist $\gamma>0$ and $\theta\in(0,1)$ such that, for all $A\in\mathcal{A}$ and all affine hyperplanes $H\subset\RR^k$,
        \[
        \mu_A(H^{(\gamma)})\leq\theta.
        \]
    \end{claim}
    \begin{proof}[Proof of Claim~\ref{c:distortedproducts}]
        Suppose the claim is false. 
        Then there are $A_n\in\mathcal{A}$ and affine hyperplanes $H_n$ such that $\mu_{A_n}(H_n^{(1/n)})>1-1/n$. The supports of $\mu_{A_n}$ lie in a fixed bounded set, so the hyperplanes $H_n$ can be assumed to meet a fixed ball. Passing to a subsequence, we have $A_n\to A\in\mathcal{A}$ and $H_n\to H$.
        For every $\varepsilon>0$, uniform convergence on $\supp(\mu)$ implies
        \[
        \supp(\mu)\cap A_n^{-1}(H_n^{(1/n)})\subseteq (A^{-1}H)^{(\varepsilon)}
        \]
        for all sufficiently large $n$. Hence $\mu((A^{-1}H)^{(\varepsilon)})=1$ for every $\varepsilon>0$, and thus $\mu(A^{-1}H)=1$. Since $A$ is invertible, $A^{-1}H$ is an affine hyperplane, contradicting our assumption on the support.
    \end{proof}

    Suppose that $M\neq U$. At least one of its defining functions, say $h$, is not identically zero on $U$. By \cite[Claim~2]{MityaginAnalytic},
    \[
    M\subseteq\{h=0\}\subseteq\bigcup_I\{\partial^Ih=0,\ \nabla\partial^Ih\neq\mathbf{0}\},
    \]
    where the union is over all multi-indices. 
    By the implicit function theorem, the sets on the right can be covered by countably many compact patches of smooth hypersurfaces. Since $\mu(M)>0$, there is a compact set $F\subset M$ of positive $\mu$-measure contained in one such patch. 
    We will work only with this fixed compact set. 
    The supports of the $\delta$-product branches have diameter less than $C_0\delta$, for a fixed $C_0$. 
    For all sufficiently small $\delta>0$ and every $\bz\in F$ there is an affine hyperplane $H_{\bz}$ such that
    \[
    F\cap B(\bz,C_0\delta)\subseteq H_{\bz}^{(\gamma\delta)}.
    \]
    If a $\delta$-product branch meets $F$, choose $\bz$ in their intersection. Its normalised measure then satisfies
    \[
    \mu_{s_B}(F)\leq\mu_{s_B}(H_{\bz}^{(\gamma\delta)})
    =\mu_A\bigl((\delta^{-1}(H_{\bz}-\bt))^{(\gamma)}\bigr)\leq\theta.
    \]
    Branches disjoint from $F$ satisfy the same bound.

    Let $(\Omega,\mathbb{P})$ be the Bernoulli process encoding the self-similar measures $\mu_i$, and let $\pi\colon\Omega\to\supp(\mu)$ be the coding map, so that the pushforward of $\mathbb{P}$ by $\pi$ is $\mu$. Put $E \coloneqq \pi^{-1}(F)$. 
    For every sufficiently small $\delta$, the product stopping cylinders form a finite partition of $\Omega$, and on each cylinder $C$ the conditional distribution of $\pi$ is a normalised $\delta$-product branch. Thus
    \[
    \mathbb{P}(E\cap C)\leq\theta\mathbb{P}(C).
    \]
    This also holds if $C$ is any finite union of cylinders. 
    Since cylinders generate the coding sigma-algebra, finite unions of cylinders approximate $E$ in measure. 
    Passing to the limit gives $\mathbb{P}(E)\leq\theta\mathbb{P}(E)$, contradicting $\mu(F)>0$. This proves that $M=U$ when $\supp(\mu)$ is not contained in an affine hyperplane. 
    Note that this argument does not require any separation condition. 

    Finally, let $L_i$ be the smallest affine subspace containing $\supp(\mu_i)$ and put $L=L_1\times\dotsb\times L_j$. Each similarity in the $i$th IFS preserves $L_i$, since it sends $\supp(\mu_i)$ into itself. 
    After identifying the positive-dimensional $L_i$ with Euclidean spaces and omitting point factors, the restricted measures are self-similar and their product has full affine span in $L$.
    If $L$ is a point, the conclusion is immediate. Otherwise, $U\cap L$ is open and convex in $L$. Restricting the defining functions to $U\cap L$ and applying the argument above, we see that $M\cap L=U\cap L$. Since $\mu(L)=1$, this proves the result.
\end{proof}

Next, we make the simple observation that the Cartesian product of affinely irreducible measures is also affinely irreducible. 

\begin{lma}\label{lma: product of linear irr}
    If $\mu_1,\dotsc,\mu_l$ are affinely irreducible Borel probability measures in $\RR^k$, then $\mu_1 \times \dotsb \times \mu_l$ is affinely irreducible in $\RR^{kl}$. 
\end{lma}
\begin{proof}
    Suppose the result does not hold, so there is some affine hyperplane $H \subset \RR^{kl}$ to which $\mu_1 \times \dotsb \times \mu_l$ gives positive measure. 
    Then there are $\mathbf{e}_1,\dotsc,\mathbf{e}_l \in \RR^k$, not all zero (say $\mathbf{e}_j \neq 0$ for some $1 \leq j \leq l$), and $t \in \RR$, such that 
    \[ H = \{(\bx_1,\dotsc,\bx_l) \in \RR^{kl} : \langle \mathbf{e}_1,\bx_1 \rangle + \dotsb + \langle \mathbf{e}_l,\bx_l \rangle = t \}. \]
    Fix all variables except $\bx_j$. The corresponding slice of $H$ is
    \[
    \left\{\bx_j\in\RR^k:\langle\mathbf{e}_j,\bx_j\rangle=t-\sum_{i\neq j}\langle\mathbf{e}_i,\bx_i\rangle\right\}.
    \]
    Since $\mathbf{e}_j\neq 0$, this is an affine hyperplane in $\RR^k$, so it has zero $\mu_j$-measure. 
    Therefore by Tonelli's theorem applied to the indicator function of $H$ and the $\bx_j$ variable, we see that $(\mu_1\times\dotsb\times\mu_l)(H)=0$, which contradicts our assumption. 
\end{proof}

\subsection{Further lemmas}
Recall the statement of Theorem~\ref{thm: Lo}. Since $\cA$ is a compact family of functions on $U$, for any compact $K\subset U\subset \mathbb{R}^k$, we consider Taylor expansions of $f\in\cA$ at $\bx\in K$. 
These are sums of monomials of integer degrees. 
\begin{lma}\label{lma: choice of D}
    Let $\cA$ be as in Theorem~\ref{thm: Lo}, so in particular $\cA$ does not contain the zero function. Let $K\subset U$ be compact. 
    Then there exist $D,C>0$ such that for all $f\in\cA$, $\bx \in K$, at least one of the coefficients of $f$ at $\bx$ with degree at most $D$ has absolute value larger than $C$.
\end{lma}
\begin{proof}
    Write $\cA$ as $\{f_i\}_{i \in I}$ for some index set $I$. Suppose that the result does not hold. Then there is a sequence $D_i\to\infty, C_i\to 0, I_i\in I, \bx_i\in K$ such that the Taylor expansion of $f_{I_i}$ at $\bx_i$ has all its coefficients for monomials with degree at most $D_i$ with absolute value at most $C_i$. As $\cA$ is a compact family, we can assume the uniform convergence of $f_{I_i}\to f\in \cA$ as well as $\bx_i\to\bx\in K$ as $i \to \infty$. 
    Consider the Taylor expansion of $f$ at $\bx$. 
    By the multivariable Cauchy integral formula~\cite[Theorem~2.2.1]{HormanderComplex}, the Taylor coefficients of $f$ at $\bx$ are the limits of the corresponding coefficients of $f_{I_i}$ at $\bx_i$, and all of these limits are zero by assumption. For this step, we need to ensure that the Cauchy integral formula can be applied for all $f_{I_i}$ over a fixed complex integral region. This is the case due to the complex-compact-open compactness. We see that the Taylor expansion of $f$ at $\bx$ is $0$, and this implies (using the connectedness of $U$) that $f$ is identically zero. 
    This contradicts the fact that $\cA$ does not contain the zero function.
\end{proof}

We can now scale the family of functions in $\cA$ such that the constant $C$ from the above lemma is $1$. 

Let $D$ be as in the above lemma. Let $H_{k,D}$ be the collection of monomials in $k$ variables with degree at most $D$ listed in some fixed order. Let $L$ be the number of such monomials. 
For each $x_1,\dotsc,x_k\in \mathbb{R}$, let
\[
H_{k,D}(x_1,\dotsc,x_k)\in \mathbb{R}^L
\]
be the evaluation of $H_{k,D}$ at $x_1,\dotsc,x_k$. 

Let $\bx_1,\dotsc,\bx_L$ be vectors in $\mathbb{R}^k$. We define $M=M(\bx_1,\dotsc,\bx_L)$ to be the real $L\times L$ matrix with row vectors
\[
H_{k,D}(\bx_1),\dotsc,H_{k,D}(\bx_L).
\]
Consider the function $h_{k,D} \colon \RR^{Lk} \to \RR$ defined by 
\[
h_{k,D} \colon (\bx_1,\dotsc,\bx_L) \mapsto \det M(\bx_1,\dotsc,\bx_L).
\]
\begin{lma}\label{l:detpropervariety}
    The set $V_{k,D} \coloneqq \{h_{k,D}=0\} \subset \RR^{Lk}$ is a proper real algebraic variety. 
\end{lma}
\begin{proof}
    It suffices to show that $h_{k,D}$ is not a constant polynomial. 
    Indeed, if it were a constant polynomial, it has to be the zero polynomial since inserting zeros in $h_{k,D}$ will give zero. Consider the image $M_H\subset\mathbb{R}^L$ of the map $H_{k,D} \colon \mathbb{R}^k\to\mathbb{R}^{L}$. 
    Since the monomials forming $H_{k,D}$ are linearly independent, $M_H$ is not contained in any homogeneous hyperplane\footnote{a \emph{homogeneous hyperplane} is an affine hyperplane passing through the origin.}, so there are linearly independent $y_1,\dotsc,y_{L}$ in $M_H$. 
    Therefore there exists a choice $\bx_1,\dotsc,\bx_{L}$ so that the matrix $M(\bx_1,\dotsc,\bx_{L})$ is not singular and so $h_{k,D}$ cannot be constant. 
\end{proof}
\begin{lma}\label{l:onebranchaway}
    There exists $\eta' > 0$ (depending only on $k,D,\mu$) such that for all $\eta \in (0,\eta')$, there exist $\eta$-branches $S_1,\dotsc,S_L$ of $\mu$, such that if $\bx_i\in S_i$ for $i\in\{1,\dotsc,L\}$, then 
    \[
    |h_{k,D}(\bx_1,\dotsc,\bx_L)|\geq \eta'. 
    \]
\end{lma}
\begin{proof}
    Since $\mu$ is affinely irreducible, $(\supp(\mu))^L$ has full affine span in $\RR^{Lk}$. 
    Applying Lemma~\ref{lma: prod of SS} to the proper algebraic variety $V_{k,D}$ from Lemma~\ref{l:detpropervariety}, we see that $\mu^L(V_{k,D})=0$. Thus there is a tuple $(\bz_1,\dotsc,\bz_L)\in\supp(\mu)^L$ for which $h_{k,D}(\bz_1,\dotsc,\bz_L)\neq0$.
    By continuity, there are $a,b>0$ such that $|h_{k,D}(\bx_1,\dotsc,\bx_L)|\geq b$ whenever $|\bx_i-\bz_i|<a$ for all $i$. For every sufficiently small $\eta>0$, choose a coding of each $\bz_i$ and stop its word at the first point where the contraction ratio is at most $\eta$. 
    Let $S_i$ be the support of the resulting branch. It contains $\bz_i$ and has diameter at most $\eta$, so every tuple $(\bx_1,\dotsc,\bx_L) \in S_1\times\dotsb\times S_L$ satisfies $|h_{k,D}(\bx_1,\dotsc,\bx_L)|\geq b$.
    Choosing $\eta'>0$ smaller than $\min\{a,b,1\}$ completes the proof. 
    The branches are chosen separately in each factor; we do not need $\mu^L$ itself to be self-similar.
\end{proof}

\subsection{Polynomials}
We first prove Theorem~\ref{thm: Lo} in the case when $\cA$ is a collection of polynomials with uniformly bounded degree. The proof already features the key ideas. 

\begin{prop}\label{p:polyloja}
Let $k,d,D \in \mathbb{N}$ and $\varepsilon > 0$, and let $\mu$ be an affinely irreducible self-similar measure in $\mathbb{R}^k$. Then there exists a positive $\beta = \beta(k,d,\mu,D,\varepsilon)$ such that the following holds. 
Let $\cP$ be any compact family of polynomials $\RR^k \to \RR^d$ with degree at most $D$, and suppose further that $\cP$ does not contain the zero polynomial. 
Then there is $c > 0$ such that for all $P\in \cP,\delta>0$, 
\[
\mu(\{|P|<\delta\}^{(\delta^\varepsilon)})\leq c\delta^{\beta}.
\]
\end{prop}
\begin{proof}
    Let $K=\supp(\mu)$. By translating, we may assume that $\mathbf{0}\in K$, and recall that $K$ has diameter at most $1$. Lemma~\ref{lma: choice of D} gives a number $C_{\cP}>0$ such that the Taylor expansion of every $P\in\cP$ at every point of $K$ has a coefficient of degree at most $D$ with norm larger than $C_{\cP}$. Here the coefficients lie in $\RR^d$; coefficients of higher degree vanish.

    Fix an IFS generating $\mu$ and let $r_{\min}$ be its smallest contraction ratio. Choose $\eta\in(0,\eta')$, where $\eta'$ is given by Lemma~\ref{l:onebranchaway}, and fix the corresponding branches $S_1,\dotsc,S_L$. Thus, for every $\bx_j\in S_j$,
    \begin{equation}\label{e:determinantlarge}
    |h_{k,D}(\bx_1,\dotsc,\bx_L)|\geq\eta'.
    \end{equation}
    Let $\mathcal{W}_{\eta}$ be the finite collection of words obtained by stopping at the first contraction ratio that is at most $\eta$. Each word has contraction ratio lying in $[r_{\min}\eta,\eta]$. The branches $S_i$ were chosen from these words. Let $p>0$ be the smallest probability weight of these words.

    Consider any branch defined by a similarity
    \[
    s(\bx)=\bt+rO\bx,\qquad O\in O_k(\RR),\quad 0<r\leq1.
    \]
    Since $\mathbf{0}\in K$, its centre $\bt=s(\mathbf{0})$ belongs to $K$. Write the Taylor expansion of $P$ at $\bt$ as
    \[
    P(\bt+\bh)=\sum_{|I|\leq D}c_I\bh^I,\qquad c_I\in\RR^d.
    \]
    Substituting $\bh=O\bx$ separately in each homogeneous degree gives coefficients $a_I\in\RR^d$ such that
    \[
    Q(\bx)\coloneqq P(s(\bx))=\sum_{|I|\leq D}r^{|I|}a_I\bx^I.
    \]
    On the finite-dimensional space of homogeneous polynomials of degree at most $D$ decomposed into homogeneous components, substitution by $O$ has an inverse given by substitution by $O^T$. These transformations act by the same scalar coefficient matrix on each output coordinate. The coefficient matrices and their inverses are uniformly bounded as $O$ ranges over the compact group $O_k(\RR)$. Consequently, there is $b=b(k,D)>0$ such that
    \[
    \max_{|I|\leq D}|a_I|\geq b\max_{|I|\leq D}|c_I|>bC_{\cP}.
    \]

    Suppose that every $\eta$-branch inside this branch meets $\{|P|<\delta\}$. In particular, we can choose $\bx_j\in S_j$ with $|Q(\bx_j)|<\delta$ for $j=1,\dotsc,L$. Let $\mathsf{C}$ be the $L\times d$ matrix whose rows are the coefficient vectors $r^{|I|}a_I$, in the order used for $H_{k,D}$, and let $\mathsf{Y}$ have rows $Q(\bx_j)$. Then
    \[
    M(\bx_1,\dotsc,\bx_L)\mathsf{C}=\mathsf{Y},\qquad
    \mathsf{C}=M(\bx_1,\dotsc,\bx_L)^{-1}\mathsf{Y}.
    \]
    By~\eqref{e:determinantlarge} and Cramer's rule, the entries of $M^{-1}$ are uniformly bounded, with a bound depending only on $k,D,\mu$. Applying this matrix to each column, or taking the norm of each row, gives
    \[
    |a_I|\ll\delta r^{-|I|}\leq\delta r^{-D}.
    \]
    Thus there is $c_*>0$, depending on the family only through $C_{\cP}$, such that whenever $r>c_*\delta^{1/D}$, at least one $\eta$-branch is disjoint from $\{|P|<\delta\}$. Its relative probability weight is at least $p$.

    Iterate the stopping words in $\mathcal{W}_{\eta}$ for
    \[
    t=\left\lfloor\frac{|\log\delta|}{2D|\log (r_{\min}\eta)|}\right\rfloor 
    \]
    generations. 
    Every branch encountered has contraction ratio at least $(r_{\min}\eta)^t\geq\delta^{1/(2D)}$, so the existence of the disjoint $\eta$-branch as described above holds when $\delta$ is sufficiently small. 
    At each generation the total weight of branches which can meet $\{|P|<\delta\}$ decreases by a factor at most $1-p$. 
    This is an iteration on symbolic cylinders, so overlapping images cause no difficulty. It follows that
    \begin{equation}\label{eq: uniform decay}
    \mu(\{|P|<\delta\})\leq(1-p)^t\ll\delta^{\beta_0},\qquad
    \beta_0=\frac{|\log(1-p)|}{2D|\log (r_{\min}\eta)|}>0.
    \end{equation}
    The exponent depends only on $k,D,\mu$. Increasing the multiplicative constant extends this estimate to all $\delta>0$.

    Finally, compactness gives a uniform Lipschitz constant $L_0\geq1$ for the functions in $\cP$ on a fixed compact neighbourhood of $K$. For sufficiently small $\delta>0$,
    \[
    K\cap\{|P|<\delta\}^{(\delta^{\varepsilon})}
    \subseteq\{|P|<\delta+L_0\delta^{\varepsilon}\}.
    \]
    Applying~\eqref{eq: uniform decay} with this threshold gives
    \[
    \mu(\{|P|<\delta\}^{(\delta^{\varepsilon})})
    \ll(\delta+L_0\delta^{\varepsilon})^{\beta_0}
    \ll\delta^{\beta_0\min\{1,\varepsilon\}}.
    \]
    This proves the result with $\beta=\beta_0\min\{1,\varepsilon\}$, after increasing the constant to cover the remaining values of $\delta$.
\end{proof}

\subsection{Analytic functions}
Next, we prove the following uniform \L ojasiewicz-type result for families of analytic functions, which gives control on the exponent $\beta$. 
\begin{prop}\label{p:quanifiedlo}
    Let $k,d,D \in \NN$ and $\varepsilon > 0$. 
    Let $\mu$ be an affinely irreducible self-similar measure on $\mathbb{R}^k$ supported inside some connected open set $U \subseteq \RR^k$. 
    Then there exists $\beta = \beta(k,d,D,\mu,U,\varepsilon) > 0$ such that the following holds. 
    Let $C>0$ and let $\cA$ be a compact family of analytic functions $U \to \RR^d$ such that for all $f \in \cA$ and $\bx \in \supp(\mu)$, at least one of the coefficients of $f$ at $\bx$ with degree at most $D$ has absolute value larger than $C$. 
    Then there exists $c>0$ such that for all $f\in \cA$ and $\delta>0$, one has
    \[
    \mu(\{|f|<\delta\}^{(\delta^{\varepsilon})}) \leq c \delta^\beta.
    \]
\end{prop} 
\begin{proof}
The proof will use many of the same ideas as the proof of Proposition~\ref{p:polyloja}. 
Let $K=\supp(\mu)$. After translating, we may assume that $\mathbf{0}\in K$, and recall that $K$ has diameter at most $1$. Fix a compact neighbourhood $K'$ of $K$ contained in $U$. By Definition~\ref{d:compact}, the holomorphic extensions of the functions in $\cA$ form a compact family on an open complex neighbourhood $V_{K'}$ of $K'$. Choose a compact complex neighbourhood of $K'$ contained in $V_{K'}$. The extensions are uniformly bounded there by compactness, so the multivariable Cauchy integral formula gives uniform bounds for the first derivatives and all derivatives of order $D+1$ on $K'$.
Consequently, there are $B\geq1$ and $\rho_0\in(0,1)$ such that Taylor expansion at any $\bt\in K$ has remainder bounded by $B\rho^{D+1}$ on $\overline{B(\bt,\rho)}$, for every $f\in\cA$ and $0<\rho<\rho_0$. We choose $\rho_0$ small enough that these balls are contained in $K'$.

Choose $\eta$, $\mathcal{W}_{\eta}$, $S_1,\dotsc,S_L$ and $p>0$ as in the polynomial proof. Consider a branch defined by
\[
s(\bx)=\bt+rO\bx,\qquad 0<r<\rho_0,\quad O\in O_k(\RR).
\]
Since $\mathbf{0}\in K$, we have $\bt\in K$. Pulling back the Taylor expansion by this full similarity gives
\[
f(s(\bx))=\sum_{|I|\leq D}r^{|I|}a_I\bx^I+R_{f,s}(\bx),\qquad
|R_{f,s}(\bx)|\leq Br^{D+1}\quad(\bx\in K),
\]
where $a_I\in\RR^d$. As in the polynomial case, the uniform bounds for coefficient transformations under $O\in O_k(\RR)$ give
\[
\max_{|I|\leq D}|a_I|>bC
\]
for some $b=b(k,D)>0$ independent of $f$ and the branch.

Suppose that every $\eta$-branch meets $\{|f|<\delta\}$. Choose $\bx_j\in S_j$ with $|f(s(\bx_j))|<\delta$. Let $\mathsf{C}$ be the $L\times d$ matrix with rows $r^{|I|}a_I$, and let $\mathsf{Y}$ and $\mathsf{R}$ have rows $f(s(\bx_j))$ and $R_{f,s}(\bx_j)$ respectively. Then
\[
M(\bx_1,\dotsc,\bx_L)\mathsf{C}=\mathsf{Y}-\mathsf{R}.
\]
The same uniform inverse-matrix bound as in the polynomial proof gives, for some $A>0$ depending only on $k,D,\mu$,
\[
|a_I|\leq A(\delta+Br^{D+1})r^{-|I|}
\leq A\delta r^{-D}+ABr.
\]
Decrease $\rho_0$ so that $AB\rho_0\leq bC/4$. If also
\[
r>c_*\delta^{1/D},\qquad c_*=(4A/(bC))^{1/D},
\]
then all coefficients have norm at most $bC/2$, a contradiction. Thus every branch with $c_*\delta^{1/D}<r<\rho_0$ has an $\eta$-branch disjoint from the sublevel set, with relative weight at least $p$.

Choose a fixed integer $N$ such that $\eta^N<\rho_0$. After $N$ generations of stopping words, every branch has ratio less than $\rho_0$. Continue for
\[
t=\left\lfloor\frac{|\log\delta|}{2D|\log (r_{\min}\eta)|}\right\rfloor
\]
more generations. The branch ratios are at least $(r_{\min}\eta)^{N+t}\geq (r_{\min}\eta)^N\delta^{1/(2D)}>c_*\delta^{1/D}$ when $\delta$ is sufficiently small. The exclusion therefore applies at every generation, and the symbolic-cylinder argument from the polynomial proof gives
\[
\mu(\{|f|<\delta\})\leq(1-p)^t\ll\delta^{\beta_0},\qquad
\beta_0=\frac{|\log(1-p)|}{2D|\log (r_{\min}\eta)|}>0.
\]
Although $N$ and the range of sufficiently small $\delta$ may depend on $\cA$ and $C$, the exponent $\beta_0$ does not. Increasing the multiplicative constant extends the bound to all $\delta>0$.

Finally, the uniform first-derivative bound on $K'$ gives a common Lipschitz constant $L_0\geq1$ near $K$. For sufficiently small $\delta$, the line segments from points of $K$ to points within distance $\delta^{\varepsilon}$ lie in $K'$. Hence
\[
K\cap\{|f|<\delta\}^{(\delta^{\varepsilon})}
\subseteq\{|f|<\delta+L_0\delta^{\varepsilon}\}.
\]
The unthickened estimate now gives
\[
\mu(\{|f|<\delta\}^{(\delta^{\varepsilon})})
\ll(\delta+L_0\delta^{\varepsilon})^{\beta_0}
\ll\delta^{\beta_0\min\{1,\varepsilon\}}.
\]
Taking $\beta=\beta_0\min\{1,\varepsilon\}$ and increasing the constant for the remaining values of $\delta$ completes the proof.
\end{proof}
\begin{proof}[Proof of Theorem~\ref{thm: Lo}]
    In the setting of the statement of Theorem~\ref{thm: Lo}, for the family $\cA$ and $K = \supp(\mu)$ we can apply Lemma~\ref{lma: choice of D} to get some integer $D$. 
    From here, Theorem~\ref{thm: Lo} follows from Proposition~\ref{p:quanifiedlo} applied with this value of $D$. 
\end{proof}

\section{Fourier decay outside sparse frequencies: Theorem \ref{thm: AFD}}\label{s:afd}

Recall the definition of Fourier decay outside sparse frequencies from Definition~\ref{d:afd}. 
For compactly supported $\mu$, $\widehat{\mu}$ is Lipschitz, so this property is equivalent, after decreasing $\delta$, to saying that $\{\bxi\in\RR^k:|\widehat{\mu}(\bxi)|\geq R^{-\delta},\ |\bxi|<R\}$ meets $\ll R^\epsilon$ dyadic cubes of unit side length, or can be covered by $\ll R^\epsilon$ balls of radius $1$. 
For the proof of Theorem~\ref{thm: AFD}, we make the following definition, giving a condition which is similar to Khalil's uniform affine non-concentration condition~\cite[(1.1)]{Khalil}. Our definition is discussed in~\cite[Remark~1.8~(2) and Remark~6.5~(1)]{Khalil}. 
\begin{defn}\label{d:inner}
    We say a Borel probability measure $\mu$ on $\RR^k$ is \emph{inner-affinely non-concentrated} if there exist $c\geq 1$ and $\phi \colon (0,1) \to \RR$ with $\phi(\varepsilon) \to 0$ as $\varepsilon \to 0$, such that for all $\bx \in \supp(\mu)$, $0 < r < 1$, $\varepsilon \in (0,1)$ and all affine hyperplanes $W \subset \RR^k$, 
    \begin{equation}\label{e:innerdef}
        \mu(W^{(\varepsilon r)} \cap B(\bx,r)) \leq \phi(\varepsilon) \mu(B(\bx,cr)),
    \end{equation}
    where $W^{(\varepsilon r)}$ is the $\varepsilon r$-neighbourhood of $W$. 
\end{defn}

We prove the following result about self-similar measures. The proof uses a somewhat similar strategy to the proof of \cite[Proposition~2.2]{FL}, and Lemma~\ref{lma: prod of SS}. 
\begin{thm}\label{thm: selfsiminner}
    Every affinely irreducible self-similar measure $\mu$ on $\RR^k$ is inner-affinely non-concentrated. 
\end{thm}

\begin{proof}
Without loss of generality we may assume that the support of $\mu$ is contained in a ball of diameter $1$. 
Let $\delta,\gamma$ be as in Lemma~\ref{l:tube} for $\mu$. 
Fix $c=2$. 
Fix an IFS defining $\mu$ and let $r_{\min} \in (0,1)$ be its smallest contraction ratio. 
Fix an affine hyperplane $K$ of $\RR^k$, and fix $\bx \in \supp(\mu)$ and $0 < r < 1$. 

\textbf{Claim:}
For all $n \geq 1$, if we let $r_n \coloneqq (c-1) r (\gamma r_{\min} / 4)^{n-1}$ then 
\[ \mu(K^{(r_n)} \cap B(\bx,r+r_n)) \leq (1-\delta)^{n-1} \mu(B(\bx,cr)).\] 

\textbf{Proof of claim:}
The proof goes by induction. The $n=1$ case holds since 
\[ K^{(r_1)} \cap B(\bx,r+r_1) \subseteq B(\bx,cr), \]
which is true because $(c-1)r \geq r_1 \geq r_2 \geq \dotsb$. 
Assume that the claim holds for some $n \geq 1$. 
Let $\Omega$ be the finite collection of words stopped at the first time that the contraction ratio is at most $r_n/2$. Then $\mu=\sum_{w\in\Omega}p_w\mu_{f_w}$, and the contraction ratios of the corresponding compositions lie in $(r_n r_{\min}/2,r_n/2]$. 
Let 
\[ \Omega' \coloneqq \{ w \in \Omega : \supp (\mu_{f_w}) \cap K^{(r_{n+1})} \cap B(\bx,r+r_{n+1}) \neq \varnothing \}.\] 
By our choice of $r_n$, if $w \in \Omega'$ then $\supp (\mu_{f_w}) \subseteq K^{(r_{n})} \cap B(\bx,r+r_{n})$. 
The pullback of $K^{(r_{n+1})}$ under $f_w$ has width at most $\gamma$, since the contraction ratio of $f_w$ is at least $r_n r_{\min}/2$. Thus by Lemma~\ref{l:tube}, 
\begin{align*}
    \mu(K^{(r_{n+1})} \cap B(\bx,r+r_{n+1})) &= \sum_{w \in \Omega'} p_w \mu(f_w^{-1}(K^{(r_{n+1})} \cap B(\bx,r+r_{n+1}))) \\
    &\leq \sum_{w \in \Omega'} p_w(1-\delta)\\
    &\leq (1-\delta) \mu(K^{(r_{n})} \cap B(\bx,r+r_{n})) \\
    &\leq (1-\delta)^{n} \mu(B(\bx,cr)),
\end{align*}
where the last inequality was by the induction hypothesis. This completes the proof of the claim. 

It follows from this claim that 
\begin{equation*}
    \mu(K^{(r_n)} \cap B(\bx,r)) \leq \mu(K^{(r_n)} \cap B(\bx,r+r_n)) \leq (1-\delta)^{n-1} \mu(B(\bx,cr)).
\end{equation*}
 We can increase the implicit multiplicative constant (in a way that depends on $k,\mu,c$) so that the desired non-concentration property holds along a geometric sequence of scales with the function $\phi(\varepsilon)$ being a constant multiple of $\varepsilon^\alpha$ for some small $\alpha > 0$ depending only on $\delta,\gamma,r_{\min}$, which in turn depend only on $k$ and $\mu$. 
 By taking $\phi(\varepsilon)$ to be an even larger constant multiple of $\varepsilon^\alpha$ we can ensure that the property holds at all scales. 
\end{proof} 

Khalil's proof of decay outside sparse frequencies under different non-concentration assumptions also goes through under the inner-affine non-concentration assumption (see \cite[Remark~1.8~(2) and Remark~6.5~(1)]{Khalil}), giving the following statement. 
\begin{thm}[\cite{Khalil}]\label{thm: innerafd}
Every compactly supported inner-affinely non-concentrated Borel probability measure on $\RR^k$ has Fourier decay outside sparse frequencies. 
\end{thm}

We now have all the pieces required for Theorem~\ref{thm: AFD}. 
\begin{proof}[Proof of Theorem~\ref{thm: AFD}]
Suppose that $\mu$ is not affinely irreducible, so $\mu(H)>0$ for some affine hyperplane $H$. 
Applying Lemma~\ref{lma: prod of SS} gives an affine subspace $L\subseteq H$ with $\mu(L)=1$, so in particular $\mu(H)=1$. 
Write $H=\{\bx:\langle\bv,\bx\rangle=a\}$, where $|\bv|=1$. For $t\in\RR$ and $\bz\in\bv^\perp$,
\[ |\widehat{\mu}(t\bv+\bz)|=|\widehat{\mu}(\bz)|. \]
Since $\widehat{\mu}$ is continuous and $\widehat{\mu}(\mathbf{0})=1$, there is $c>0$ such that this quantity is at least $1/2$ whenever $|\bz|<c$. The volume of such frequencies inside $B(\mathbf{0},R)$ is $\gtrsim R$, so $\mu$ cannot have Fourier decay outside sparse frequencies. 
The converse implication follows from Theorems~\ref{thm: selfsiminner} and~\ref{thm: innerafd}. 
\end{proof}

\section{Self-conformal measures}\label{s:conformal}
\subsection{Self-conformal measures in the plane}\label{ss:selfconfplane}
Recall the definition of self-conformal measures from Section~\ref{ss:selfconfprelim}. For now we work in the (complex) plane. 
For brevity we introduce the following terminology. 
\begin{defn}
    A set $\Gamma \subset \CC$ is an \emph{analytic curve} if it is an embedded $1$-dimensional real-analytic submanifold of $\CC$. In other words, for each $z_0 \in \Gamma$, there exists an open subset $U \subset \CC$ with $z_0 \in U$ and a map $\gamma \colon (0,1) \to \CC$ whose real and imaginary components are both real-analytic, such that $\gamma$ is a homeomorphism from $(0,1)$ onto $\Gamma \cap U$, and $\gamma'(t) \neq 0$ for all $t \in (0,1)$. 
\end{defn}
\begin{defn}\label{def: curvereducible}
    Let $\Phi=\{\varphi_i\colon U\to U\}_i$ be a conformal IFS and $\nu$ an associated self-conformal measure. 
    \begin{itemize}
        \item We say $\Phi$ and $\nu$ are \emph{nonlinear} if there exists $i$ such that $\varphi_i$ is non-affine. 
        \item We say $\Phi$ and $\nu$ are \emph{curve-reducible} if $\supp(\nu)$ is contained in an analytic curve $\Gamma\subset U$, and \emph{curve-irreducible} otherwise. 
    \end{itemize}
\end{defn}

Since the support of every associated measure with positive weights is the attractor, curve-reducibility does not depend on the weights.

A key goal of this section is the following result, which implies Theorem~\ref{thm: analyticconformal} and will follow by combining Theorem~\ref{thm: conjselfconf} below with a recent result of Algom, Rodriguez~Hertz and Wang~\cite{AHW25}. 
\begin{thm*}
    Let $\nu$ be a nonlinear and curve-irreducible self-conformal measure, and assume the domain $D \subset \CC$ of the contractions is a closed ball. 
    Then $\nu$ has polynomial Fourier decay. 
\end{thm*}
Self-similar measures like the Cantor--Lebesgue measure (or products of this measure with itself) may not be Rajchman; this is the reason for the nonlinearity assumption. 
Example~\ref{ex:pushtoparabola} below will show why the curve-irreducibility assumption cannot be removed. 
Note that if $\Haus\supp(\nu)>1$ then $\nu$ is automatically curve-irreducible. 

The theory of nonlinear self-conformal measures breaks into two cases according to the following definition. 
\begin{defn}\label{def:conj}
    We say that a conformal IFS $\Phi = \{\varphi_i\}_{i \in I}$ on $\CC$ is \emph{conjugate to linear} if there exists a self-similar IFS $\{\psi_i\}_{i \in I}$ and a holomorphic diffeomorphism $f$ from an open neighbourhood of the self-similar set to an open neighbourhood of the self-conformal set, so that $\varphi_i = f \circ \psi_i \circ f^{-1}$ for all $i \in I$. 
\end{defn}
We will see that the tools in this paper prove very fruitful in the conjugate-to-linear case, and address point~2 below the statement of Theorem~1.1 in~\cite{AHW25}. 

\begin{lma}\label{l:gatherpushforward}
Let $\nu$ be a self-conformal measure in $\RR^k$ with weights $p_i$ for an IFS $\{\varphi_i\}$ which is conjugate to a self-similar IFS $\{\psi_i\}$ via a conformal diffeomorphism $f$, using notation as above. 
\begin{enumerate}
    \item\label{i:push} If $\mu$ is the self-similar measure for $\{\psi_i\}$ with the same weights $p_i$, then $\nu = \mu_f$ (in this case we say $\nu$ and $\mu$ are conjugate). 
    \item\label{i:nonlinear} If $\nu$ is a nonlinear self-conformal measure, then $f$ is non-affine on every connected component of its domain which intersects $\supp(\mu)$. 
\end{enumerate}
\end{lma}
\begin{proof}
\eqref{i:push}: For all Borel $A \subseteq \RR^k$, 
\begin{align}\label{e:conjispush} 
\begin{split}
		\mu_f (A) = \mu(f^{-1} (A)) = \sum_{i} p_i \mu(\psi_{i}^{-1}\circ f^{-1} (A)) &= \sum_{i} p_i \mu (f^{-1}\circ\varphi_i^{-1} (A))\\*
		&=\sum_{i} p_i (\mu_f)_{\varphi_{i}} (A). 
\end{split}
\end{align} 
By the definition of a self-conformal measure this shows that $\nu=\mu_f$. 

\eqref{i:nonlinear}: Denote a non-affine map by $\varphi_1$. Suppose for a contradiction that $f$ is affine on a connected component $C$ of its domain which intersects $\supp(\mu)$. Choose $z\in f(C)\cap\supp(\nu)$ and $\delta>0$ such that $B(z,\delta)\subset f(C)$ and is contained in the connected domain of the conformal IFS. Choose a coding $(i_n)_{n=1}^\infty$ for $z$. For a sufficiently long prefix $w=(i_1,\dotsc,i_n)$, the image $\varphi_w(\supp(\nu))$ is contained in $B(z,\delta)$. The same holds for $\varphi_w\circ\varphi_1(\supp(\nu))$, since $\varphi_1$ preserves the attractor. By conjugacy, both $\psi_w(\supp(\mu))$ and $\psi_w\circ\psi_1(\supp(\mu))$ are contained in $C$.

Choose $x\in C\cap\supp(\mu)$ and a sufficiently small open neighbourhood $V\subset C$ of $x$ such that $f(V)$ lies in the domain of the IFS and $\psi_w(V)$ and $\psi_w\circ\psi_1(V)$ are contained in $C$. On these sets $f$ is affine, and its inverse on $f(C)$ is affine. Therefore both $\varphi_w=f\circ\psi_w\circ f^{-1}$ and $\varphi_w\circ\varphi_1=f\circ\psi_w\circ\psi_1\circ f^{-1}$ are affine on the non-empty open set $f(V)$. By the identity theorem they are affine on the connected domain of the IFS. Since $\varphi_w$ is injective, its affine inverse exists, so $\varphi_1$ is affine as well. This contradicts our choice of $\varphi_1$ and proves~\eqref{i:nonlinear}. The same argument applies to conformal conjugacies in $\RR^k$, using the real-analytic identity theorem.
\end{proof}

We will use the following fact from complex analysis. 
\begin{lma}\label{lma:cauchyriemann}
    Let $U \subset \CC$ be non-empty, open and connected and $f \colon U \to \CC$ holomorphic. 
    Then $f$ is degenerate if and only if $f$ is affine. 
\end{lma}
\begin{proof}
    The backward implication is trivial, so we prove the forward implication. 
    Assume $f$ is degenerate. 
    Letting $u,v$ be the real and imaginary parts of $f$ and recalling Proposition~\ref{prop: degequivalences}, the degeneracy assumption implies that there is some nontrivial affine form $L$ so that 
    \[ L(x,y,u(x,y),v(x,y))=0 \] 
    identically on $U$. 
    We write out the affine form with real coefficients (not all zero) as
\[
c_0+c_x x+c_y y+ c_u u+c_v v=0.
\]
We can take partial derivatives with respect to $x,y$ and see that
\begin{align*}
\begin{bmatrix}
c_u & -c_v\\
c_v & c_u
\end{bmatrix}
\begin{bmatrix}
u_x \\
u_y 
\end{bmatrix}=
\begin{bmatrix}
-c_x \\
-c_y 
\end{bmatrix},
\end{align*}
where we have used the Cauchy--Riemann equations ($u_x=v_y$, $u_y=-v_x$) for holomorphic functions. 
The determinant of the matrix cannot be zero unless $c_u=c_v=0$; in this case, $L$ indicates an affine relation between $x,y$ which cannot exist. Therefore we can solve the above linear equations and see that $u_x,u_y$ are constant functions. This implies that $u$ is affine. Similarly, $v$ is affine. Thus $f$ is affine, as required. 
\end{proof} 

We can now prove the following Kaufman-type result for holomorphic pushforwards. 
\begin{thm}\label{thm:analytickaufman}
    Let $\mu$ be an affinely irreducible self-similar measure on $\CC$, let $U \subset \CC$ be a connected open neighbourhood of $\supp(\mu)$, and let $f \colon U \to \CC$ be holomorphic and non-affine. Then $\mu_f$ has polynomial Fourier decay. 
\end{thm}
\begin{proof}
    By Lemma~\ref{lma:cauchyriemann}, $f$ is non-degenerate, and $\mu$ is non-expanding since it lies in the plane, so $\mu_f$ has polynomial Fourier decay by Theorem~\ref{thm: maingeneraldomain}~\eqref{i:nonexpanding}. 
\end{proof}

\begin{rem}
    In \cite[Theorem~1.6]{BY-quantitative}, we proved quantitative polynomial Fourier decay for nonlinear holomorphic pushforwards of self-similar measures in $\CC$ with dimension greater than $1$. Theorem~\ref{thm:analytickaufman}, on the other hand, does not quantify the exponent, but works even for pushforwards of self-similar measures with small dimension. 
\end{rem}

Next, we record the relation between affine irreducibility of the self-similar measure and the image being contained in an analytic curve. 

\begin{lma}\label{lma: curvesandlines}
    Let $\mu$ be a self-similar measure on $\CC$ and let $f$ be a holomorphic diffeomorphism on an open neighbourhood $U$ of $\supp(\mu)$. Then the following are equivalent:
    \begin{enumerate}
        \item\label{i:notaffineirred} $\mu$ is affinely reducible;
        \item\label{i:containstraightline} $\supp(\mu)$ is contained in a straight line;
        \item\label{i:containanalyticcurve} $\supp(\mu_f)$ is contained in an analytic curve;
        \item\label{i:posmeascurve} $\mu_f$ gives positive measure to some analytic curve.
    \end{enumerate}
\end{lma}
\begin{proof}
    If $\mu$ gives positive measure to a line, then Lemma~\ref{lma: prod of SS} shows that its whole support lies in that line. This proves \eqref{i:notaffineirred}~$\implies$~\eqref{i:containstraightline}, and \eqref{i:containstraightline}~$\implies$~\eqref{i:notaffineirred} is immediate. If the support lies in a line $L$, then $f(L\cap U)$ is an embedded analytic curve containing $\supp(\mu_f)$. Thus \eqref{i:containstraightline}~$\implies$~\eqref{i:containanalyticcurve}, and clearly \eqref{i:containanalyticcurve}~$\implies$~\eqref{i:posmeascurve}. 

    It remains to show \eqref{i:posmeascurve}~$\implies$~\eqref{i:containstraightline}. 
    To this end, assume \eqref{i:posmeascurve}, so $\mu_f(\Gamma)>0$ for some analytic curve $\Gamma$. 
    Cover $\Gamma\cap f(U)$ by countably many curve patches. On one such patch there is an open set $W\subset f(U)$ and a nonzero real-analytic function $h\colon W\to\RR$ whose zero set is $\Gamma\cap W$, and $\mu_f(\Gamma\cap W)>0$. We may take $\nabla h\neq\mathbf{0}$ on $W$. Cover $f^{-1}(W)$ by countably many balls and choose one, denoted by $B$, on which the zero set of $h\circ f$ has positive $\mu$-measure. By inner regularity, choose a compact subset $F$ of this zero set inside $B$ with $\mu(F)>0$. 
    Let $\{\psi_i\}$ be a self-similar IFS generating $\mu$. 
    Consider a stopping 
    \[ \mu=\sum_{w\in\Omega}p_w\mu_{\psi_w} \]
    for the self-similar measure at a scale small enough that every cylinder meeting $F$ is contained in $B$. 
    Since $\mu(F)>0$, some branch $\mu_{\psi_w}$ has positive mass on $F$ and its support lies inside $B$. If $\mu$ were affinely irreducible in $\CC$, then this branch would be too. 
    Applying Lemma~\ref{lma: prod of SS} to $\mu_{\psi_w}$ and $h\circ f$ on the (convex) ball $B$ would then force $h\circ f$ to vanish identically on $B$. This is impossible, since $f(B)$ is open and $h$ defines a proper curve patch. Hence $\supp(\mu)$ lies in a line and~\eqref{i:containstraightline} holds, as required. 
\end{proof}

In particular, a conjugate-to-linear self-conformal measure in the plane is curve-irreducible if and only if its corresponding self-similar measure is affinely irreducible. 

We are ready to apply Theorem~\ref{thm:analytickaufman} to prove our main contribution of this section. 
\begin{thm}\label{thm: conjselfconf}
    Every curve-irreducible, nonlinear, and conjugate-to-linear self-conformal measure $\nu$ in $\CC$ has polynomial Fourier decay. 
\end{thm}
\begin{proof}
    By Lemma~\ref{l:gatherpushforward}, write $\nu=\mu_f$, where $\mu$ is self-similar and $f$ is a holomorphic diffeomorphism. By Lemma~\ref{lma: curvesandlines}, $\mu$ is affinely irreducible. Also, $f$ is non-affine on every connected component of its domain meeting $\supp(\mu)$. 

    Cover $\supp(\mu)$ by finitely many balls whose closures lie in the domain of $f$. For a sufficiently large fixed $N$, every cylinder $\psi_w(\supp(\mu))$ with $|w|=N$ is contained in one of these balls, say $B_w$. The map $f\circ\psi_w$ is holomorphic on the connected open set $\psi_w^{-1}(B_w)$ containing $\supp(\mu)$. 
    By the identity theorem, $f\circ\psi_w$ is non-affine on $\psi_w^{-1}(B_w)$. 
    Thus Theorem~\ref{thm:analytickaufman} gives polynomial Fourier decay for each $\mu_{f\circ\psi_w}$. There are finitely many such words, and
    \[ \nu=\sum_{|w|=N}p_w\mu_{f\circ\psi_w}. \]
    Taking the smallest of their decay exponents gives the required polynomial Fourier decay for $\nu$.
\end{proof}

To combine our result with~\cite{AHW25}, we need a rigidity result that allows us to upgrade from a smooth conjugacy to a holomorphic one. 
We first record the local linearisation used in this argument.

\begin{lma}\label{lma: locallinearisation}
    Let $\varphi$ be holomorphic in a neighbourhood of some point $z_0 \in \CC$, with $\varphi(z_0)=z_0$ and $0<|\varphi'(z_0)|<1$. 
    Then there is a holomorphic diffeomorphism $\kappa$ on a neighbourhood of $z_0$ such that $\kappa(z_0)=0$, $\kappa'(z_0)=1$, and
    \[ \kappa\circ\varphi=\varphi'(z_0)\kappa \]
    on a sufficiently small neighbourhood of $z_0$.
\end{lma}
\begin{proof}
    Translating, we may assume without loss of generality that $z_0=0$. 
    Choose $\rho$ with $|\varphi'(z_0)|<\rho<\sqrt{|\varphi'(z_0)|}$. On a sufficiently small disc, Taylor's theorem gives that $|\varphi(z)|\leq\rho|z|$ and $\varphi(z)= \varphi'(z_0) z+O(|z|^2)$. The holomorphic functions $\kappa_n(z) \coloneqq \varphi'(z_0)^{-n}\varphi^n(z)$ satisfy
    \[
    |\kappa_{n+1}(z)-\kappa_n(z)|
    \leq\frac{C|z|^2}{|\varphi'(z_0)|}\left(\frac{\rho^2}{|\varphi'(z_0)|}\right)^n.
    \]
    Hence they converge uniformly on a smaller disc to some holomorphic function $\kappa$. Since $\kappa_n'(0)=1$ for every $n$, we have $\kappa'(0)=1$. 
    Passing to the limit gives $\kappa\circ\varphi = \varphi'(z_0)\kappa$, and the inverse function theorem gives that $\kappa$ is a local diffeomorphism, as required. 
\end{proof}

\begin{lma}\label{lma: simultaneousconjugacy}
    Let $\Phi=\{\varphi_i\}_{i\in I}$ be a holomorphic conformal IFS. Suppose it is conjugate to a self-similar IFS $\Psi=\{\psi_i\}_{i\in I}$ by a real $C^1$ diffeomorphism $h$ from an open neighbourhood of the self-similar set to an open neighbourhood of the self-conformal set. 
    Then $\Phi$ is conjugate to a self-similar IFS by a \emph{holomorphic} diffeomorphism between open neighbourhoods of the attractors.
\end{lma}
\begin{proof}
    The key idea of the proof is to factor the $C^1$ conjugacy $h$ as $h = f \circ T$, where $T$ is real affine and $f$ is holomorphic. 

    Let $K$ be the self-similar set. 
    By decreasing the domain of $h$, we may take it to be a bounded open neighbourhood $K^{(\varepsilon)}$ of $K$ on which all the conjugacy identities hold and whose image lies in the domain $U$ of $\Phi$. 
    Moreover, every $\psi_i$ maps $K^{(\varepsilon)}$ into itself, since it maps $K$ into $K$ and has contraction ratio less than $1$.

    Let $w_0\in K$ be the fixed point of $\psi_1$, write
    \[ \psi_1(w) = w_0 + r_1 O_1(w-w_0) \]
    for some $O_1\in O_2(\RR)$, and note that $\varphi_1(h(w_0))=h(w_0)$. Let $\lambda \coloneqq \varphi_1'(h(w_0))$ and let $\kappa$ be the local holomorphic linearisation from Lemma~\ref{lma: locallinearisation}. 
    On a small ball centred at $w_0$, consider the invertible real-linear derivative map $D(\kappa\circ h)(w_0)$. 
    Differentiating the identity $\kappa\circ h\circ\psi_1=\lambda (\kappa\circ h)$ at $w_0$ gives
    \begin{equation}\label{e:derividentity} 
    D(\kappa\circ h)(r_1 O_1) = \lambda D(\kappa\circ h). 
    \end{equation}
    Comparing determinants shows that $O_1$ is orientation preserving and $r_1=|\lambda|$.

    Now define $T(w) \coloneqq D(\kappa\circ h)(w-w_0)$. 
    The identity~\eqref{e:derividentity} gives $T\circ\psi_1=\lambda T$. 
    For all $w$ in a sufficiently small ball centred at $w_0$, Taylor's theorem gives 
    \[ \lambda^n (\kappa\circ h)(w) = \kappa\circ h(\psi_1^n(w)) = T(\psi_1^n(w))+o(r_1^n) = \lambda^n T(w)+o(r_1^n). \]
    Dividing by $\lambda^n$ and letting $n\to\infty$ shows that $\kappa\circ h(w)=T(w)$, and hence $h=\kappa^{-1}\circ T$, throughout this ball. 
    
    Since $K^{(\varepsilon)}$ is bounded, choose $N$ such that $\psi_1^N(K^{(\varepsilon)})$ lies in this small ball. On $T(K^{(\varepsilon)})$, define $f \coloneqq h\circ T^{-1}$. 
    For $z\in T(K^{(\varepsilon)})$ we have 
    \begin{equation}\label{e: extendholomorphicconjugacy}
    f(z)=(\varphi_1^N)^{-1}\bigl(\kappa^{-1}(\lambda^Nz)\bigr).
    \end{equation} 
    This formula is well-defined: if $z=T(w)$ with $w\in K^{(\varepsilon)}$, then
    \[ \kappa^{-1}(\lambda^Nz)=h(\psi_1^N(w)) = \varphi_1^N(h(w))\in\varphi_1^N(U). \]
    Since $\varphi_1^N$ is injective and holomorphic on $U$, its inverse is holomorphic on $\varphi_1^N(U)$. Consequently~\eqref{e: extendholomorphicconjugacy} makes $f$ holomorphic on the whole of $T(K^{(\varepsilon)})$, even if $K^{(\varepsilon)}$ is disconnected. Also, $f=h\circ T^{-1}$ is a real diffeomorphism, so we have shown that it is a holomorphic diffeomorphism on $T(K^{(\varepsilon)})$. 

    For each $i$, the real affine map $\overline{\psi_i} \coloneqq T\circ\psi_i\circ T^{-1}$ satisfies
    \[ \overline{\psi_i} = f^{-1}\circ\varphi_i\circ f \]
    on $T(K^{(\varepsilon)})$. Thus $\overline{\psi_i}$ is holomorphic on a non-empty open set. Since it is a real affine map, it follows that it is a complex affine map everywhere. 
    Moreover, its contraction ratio is the same as that of $\psi_i$. 
    Hence the maps $\overline{\psi_i}$ form a self-similar IFS, and $\varphi_i = f\circ \overline{\psi_i} \circ f^{-1}$ for each $i$. 
    The attractor of $\{\overline{\psi_i}\}$ is $T(K)$, so $f$ is defined on a neighbourhood of the whole attractor, as required.
\end{proof}

We note that the original $C^1$ conjugacy need not itself be holomorphic. 
In the proof of Lemma~\ref{lma: simultaneousconjugacy}, the real affine change $T$ gives a holomorphic conjugacy to another self-similar IFS.

We now have all the pieces to combine with~\cite{AHW25} and give Theorem~\ref{thm: analyticconformal}. 
\begin{proof}[Proof of Theorem~\ref{thm: analyticconformal}]
If $\nu$ is conjugate to linear then the result follows from Theorem~\ref{thm: conjselfconf}.
Otherwise, Lemma~\ref{lma: simultaneousconjugacy} shows that there is no $C^2$ (or even $C^1$) diffeomorphism conjugating the IFS to a self-similar IFS on a neighbourhood of the attractor. The curve-irreducibility assumption means that the attractor is not contained in an analytic planar curve. 
After rescaling the closed ball to the unit disc, polynomial Fourier decay follows from~\cite[Theorem~1.1]{AHW25}.\footnote{This result assumes that the set $D$ in the definition of the self-conformal measure from Section~\ref{ss:selfconfprelim} is a closed disc, but it is likely that this assumption can be relaxed, as described in the text after the statement of the theorem.}
\end{proof}

For self-conformal measures supported inside analytic curves, one needs to be more careful. 
If the self-conformal measure is supported inside a straight line, then there is no Fourier decay along directions orthogonal to that line (irrespective of any nonlinearity conditions). 
More interestingly, consider the following example. Note that $z \mapsto z + i(z+1)^2$ is a non-degenerate function $\CC \to \CC$ (i.e. $\RR^2 \to \RR^2$), but that its restriction to the real line is the degenerate map $x\mapsto(x,(x+1)^2)$ from $\RR\to\RR^2$, which is the example from Remark~\ref{rem:degimplications}~\eqref{i:liftdeg}. 
\begin{exm}\label{ex:pushtoparabola}
    Let $\mu$ be the Cantor--Lebesgue measure on $[0,1] \subset \CC$, i.e. the $(1/2,1/2)$ self-similar measure for $\{\varphi_1,\varphi_2\}$ where $\varphi_1(z) = z/3$ and $\varphi_2(z) = (z+2)/3$. 
    Fix a complex neighbourhood $U$ of $[0,1]$ that is small enough that $f \colon U \to f(U)$ is a conformal diffeomorphism, where $f$ is the quadratic polynomial $f(z) = z + i (z+1)^2$, noting that $f$ lifts $[0,1]$ to a parabola. We can moreover choose $U$ in such a way that $2 < |f'(z)| < 5$ for all $z \in U$ and that $\varphi_1$ and $\varphi_2$ map $U$ into itself. 
    The IFS 
    \[ \{ f \circ \varphi_1 \circ f^{-1},f \circ \varphi_2 \circ f^{-1} \} \] 
    on $f(U)$ can be seen to be uniformly contracting by a simple application of the chain rule. 
    By the calculation from~\eqref{e:conjispush}, $\mu_f$ is the self-conformal measure for this IFS and $(1/2,1/2)$ weights.\footnote{Incidentally, $\mu_f$ is also a self-affine measure for an IFS consisting of two affine contractions, with the same weights $(1/2,1/2)$, see \cite[Lemma~2.10]{AK}. The point is that the affine and conformal maps coincide for input values on the parabola, but are different away from the parabola.}
    Moreover, $\widehat{\mu_f}(3^n) = \widehat{\mu}(3^n)$ takes a constant nonzero value independent of $n \in \NN$, so $\mu_f$ is not Rajchman.
\end{exm}

More generally, if $\mu$ is a self-similar measure supported on $[0,1]$ without polynomial Fourier decay, $U$ is a complex neighbourhood of $[0,1]$, and $f \colon U \to \CC$ is a holomorphic map which degenerates along the line $[0,1]$, then it is not difficult to see that $\mu_f$ does not have polynomial Fourier decay. 
On the other hand, \cite[Theorem~1.1]{AK} implies that as long as $f([0,1])$ is not contained in a line, these pushforward measures have polynomial Fourier decay outside a very sparse set of exceptional frequencies. 

\begin{rem}
    Example~\ref{ex:pushtoparabola} appears to contradict Mosquera and Olivo's \cite[Theorem~3.1]{MosqueraOlivo} as stated. However, their proof uses \cite[Proposition~2.2]{MosqueraOlivo}, which requires a non-real common multiplier for the contractions, thus ruling out Example~\ref{ex:pushtoparabola}.
\end{rem}

The next result gives sufficient conditions for polynomial Fourier decay for nonlinear, conjugate-to-linear, curve-\emph{reducible} self-conformal measures. 

\begin{prop}\label{p:conjcurve}
    Let $\nu$ be a nonlinear, curve-reducible self-conformal measure in $\CC$ which is conjugate to linear and gives zero measure to straight lines. Then $\nu$ has polynomial Fourier decay unless there exists a self-similar measure $\mu$ supported in $[0,1]$ such that 
        \begin{equation}\label{e:nonpolyalongreals} 
        \limsup_{\substack{\xi \in \RR \\ |\xi| \to \infty}} \frac{|\widehat{\mu}(\xi)|}{|\xi|^{-\sigma}} = \infty 
        \end{equation}
        for all $\sigma > 0$, some complex neighbourhood $U$ of $\supp(\mu)$, and a conformal map $g \colon U \to \CC$ whose restriction to $[0,1] \cap U$ is degenerate, such that $\nu = \mu_g$. 
\end{prop}
\begin{proof}
    As in the proof of Theorem~\ref{thm: conjselfconf}, write $\nu=\mu_f$ for a self-similar measure $\mu$ and a holomorphic diffeomorphism $f$ on a neighbourhood $U$ of $\supp(\mu)$. Since $\nu$ is curve-reducible, Lemma~\ref{lma: curvesandlines} gives that $\supp(\mu)$ lies in a line. After a complex affine change of coordinates, we may regard $\mu$ as an (affinely irreducible) self-similar measure on $\RR$ supported in $[0,1]$, with $\nu=\mu_f$ for the (correspondingly changed) holomorphic conjugacy. 

    Let $J$ be the union of the connected components of $U\cap\RR$ which meet $\supp(\mu)$. 
    Let $I$ be a connected component of $J$. 
    If $f_{J}$ mapped an interval inside $I$ into a line $H$, a real-analytic defining equation for $H$ would vanish on $I$, so $g(I)\subseteq H$. 
    But $\mu(I)>0$, so $\nu(H)>0$, which is a contradiction. 
    Hence $g$ is non-trapped on $J$.

    If~\eqref{e:nonpolyalongreals} fails for some $\sigma>0$, then $\mu$ has polynomial Fourier decay on $\RR$. Theorem~\ref{thm: maingeneraldomain}, applied to $f_{J}$ under assumption~\eqref{i:selfsimhaspoly}, then gives polynomial Fourier decay for $\nu$. Likewise, if $f_{J}$ is non-degenerate, then Theorem~\ref{thm: maingeneraldomain}~\eqref{i:nonexpanding} gives polynomial Fourier decay, since $\mu$ is a non-expanding self-similar measure on $\RR$. 
    Therefore if $\nu$ does not have polynomial Fourier decay, then~\eqref{e:nonpolyalongreals} holds for every $\sigma>0$ and $f_{J}$ is degenerate on some interval in $J$. 
    Thus $f|_{[0,1]\cap U}$ is degenerate, as required.
\end{proof}

Theorem~\ref{thm: conjselfconf} gives polynomial Fourier decay for every nonlinear, curve-irreducible self-conformal measure which is conjugate to linear. Proposition~\ref{p:conjcurve} also gives polynomial Fourier decay in the curve-reducible case unless the conjugate self-similar measure lacks polynomial Fourier decay and the restricted conjugacy degenerates. 

\subsection{Self-conformal measures in higher dimensions}
In $\RR^k$, $k \geq 3$, as mentioned in Section~\ref{ss:selfconfprelim}, conformal maps are M\"obius transformations. In particular one can locally write a non-affine conformal map as
\begin{equation}\label{e:conformalhigherform}
    f(\bx)=\bb+\frac{\alpha A(\bx-\ba)}{|\bx-\ba|^2},
    \qquad \alpha\neq0,\quad A\in O_k(\RR).
\end{equation}
The point $\ba$, called a \emph{pole}, does not belong to the domain on which~\eqref{e:conformalhigherform} represents $f$. 
Before proving our results on self-conformal measures in higher dimensions, we need two lemmas. 

\begin{lma}\label{lma: inversionaffinehull}
    Let $L$ be an $m$-dimensional affine subspace of $\RR^k$, and let $f$ have the form~\eqref{e:conformalhigherform}. If $\ba\in L$, then $f(L\setminus\{\ba\})$ is contained in an $m$-dimensional affine subspace. If $\ba\notin L$, then $f(L)$ is contained in an $m$-sphere which is in turn contained in some $(m+1)$-dimensional affine subspace. 
\end{lma}
\begin{proof}
    Write $L=\ba+\bc+V$, where $V$ is an $m$-dimensional linear subspace and $\bc\in V^\perp$. For $\bz\in V$ with $\bc+\bz\neq\mathbf{0}$,
    \begin{equation}\label{e:hyperplanemobius}
    f(\ba+\bc+\bz)=\bb+\alpha A\frac{\bc+\bz}{|\bc|^2+|\bz|^2}. 
    \end{equation}
    If $a \in L$ then $\bc=\mathbf{0}$, and the point from~\eqref{e:hyperplanemobius} lies in $\bb+AV$. 
    Otherwise, the point from~\eqref{e:hyperplanemobius} lies in 
    \[ \{ \bb+A(\bv + t\bc) : \bv \in V, t \in \mathbb{R} \},\] 
    which has dimension $m+1$. Moreover, 
    \[
    \left|\frac{\bc+\bz}{|\bc|^2+|\bz|^2}-\frac{\bc}{2|\bc|^2}\right|^2
    =\frac{1}{4|\bc|^2}.
    \]
    Thus the image lies on the sphere with centre $\bb+\alpha A\bc/(2|\bc|^2)$ and radius $|\alpha|/(2|\bc|)$ in the indicated affine subspace.
\end{proof}

We next show that such maps $f$ cannot be conical. The subtle case happens when the self-similar measure $\mu$ is supported in a proper affine subspace. In this case, non-conicality must be checked after restricting $f$ to that affine subspace, which does not follow immediately from non-conicality on $\RR^k$. 

\begin{lma}\label{lem:highernonconical}
    Let $f$ have the form~\eqref{e:conformalhigherform}. Then $f$ is non-conical on every open subset of $\RR^k\setminus\{\ba\}$. 
    Moreover, if $L$ is an affine hyperplane with $\ba\notin L$, then $f|_L$ is also non-conical on every non-empty relatively open subset of $L$.
\end{lma}
\begin{proof}
    Translating the source, we may take $\ba=\mathbf{0}$ without loss of generality. 
    Letting $\bv\in\mathbb{S}^{k-1}$, on the full-dimensional domain we have 
    \[
    P_\bv(\bx)=\alpha\frac{|\bx|^2 A^T\bv -2\langle A^T\bv ,\bx\rangle\bx}{|\bx|^4},
    \qquad |P_\bv(\bx)|=\frac{|\alpha|}{|\bx|^2}.
    \]
    The latter function is not constant on any non-empty open set, proving the non-conicality of $f$. 

    To prove the non-conicality of $f|_L$, write $L=h\bn+V$, where $h>0$, $|\bn|=1$ and $V=\bn^\perp$. Decompose $A^T\bv = \bu+t\bn$ with $\bu\in V$ and $t\in\RR$. 
    For $\bz \in V$, write $\langle \bv, f(h \bn + \bz) \rangle = \langle \bv, \bb \rangle + F(\bz)$, where 
    \[ F(\bz) \coloneqq \alpha\frac{\langle\bu,\bz\rangle+th}{h^2+|\bz|^2}. \]
    The gradient of $F$ along the vector space $V$ is 
    \begin{equation}\label{e: tangentialinversion}
    \nabla_V F(\bz)=\alpha\frac{(h^2+|\bz|^2)\bu-2(\langle\bu,\bz\rangle+th)\bz}{(h^2+|\bz|^2)^2}.
    \end{equation}
    Suppose that $|\nabla_V F|$ were constant on a non-empty open subset of $V$. Since $|\nabla_V F|^2$ is real analytic on the whole of $V$, the identity theorem would mean that it had to be constant everywhere. But~\eqref{e: tangentialinversion} tends to $\mathbf{0}$ as $|\bz|\to\infty$, so this constant would be zero. Hence $F$ would be constant on $V$. Since $F(\bz)\to0$ as $|\bz|\to\infty$, it would vanish identically. This gives $\bu=\mathbf{0}$ and $t=0$, so 
    \[ 0 = |A^T\bv|^2 = \bv^T A A^T \bv = \bv^T \bv = 1,\] 
    which is a contradiction. 
    Thus $f|_L$ is non-conical.
\end{proof}

Self-conformal measures in higher dimensions can be defined analogously as in the two-dimensional setting introduced in Section~\ref{ss:selfconfprelim}. 
We say that a self-conformal measure $\nu$ for an IFS $\{\varphi_i\}$ on $\RR^k$ is \emph{conjugate to linear} if there is a self-similar IFS $\{\psi_i\}$ and a conformal diffeomorphism $f$ from an open neighbourhood $U$ of the self-similar set to an open neighbourhood of $\supp(\nu)$ such that $\varphi_i=f\circ\psi_i\circ f^{-1}$ for all $i$.
The final result of this section is given below.

\begin{thm}\label{thm: higherconformal}
    Fix an integer $k\geq3$ and let $\nu$ be a self-conformal measure on $\RR^k$. Assume that
    \begin{enumerate}
        \item $\nu$ is conjugate to linear,
        \item\label{i:hyperplanemeasure0} $\nu(H)=0$ for all affine hyperplanes $H\subset\RR^k$, and
        \item at least one of the conformal maps in the IFS is non-affine.
    \end{enumerate}
    Then $\nu$ has polynomial Fourier decay.
\end{thm}
\begin{proof}
    By Lemma~\ref{l:gatherpushforward}, write $\nu=\mu_f$, where $\mu$ is self-similar and $f\colon U\to f(U)$ is a conformal diffeomorphism. 
    We can of course ignore any connected components of $U$ which do not meet $\supp(\mu)$. 
    Every remaining component $C$ has $\mu(C)>0$, and $f$ is non-affine of $C$ with the form~\eqref{e:conformalhigherform} (with parameters which may depend on $C$). 

    Let $L$ be the smallest affine subspace containing $\supp(\mu)$ and write $m=\dim L$. 
    If $m\leq k-2$, then Lemma~\ref{lma: inversionaffinehull} shows that $f(C\cap L)$ lies in an affine hyperplane $H$ of $\RR^k$. 
    It follows that $\nu(H)\geq\mu(C)>0$, which is a contradiction. Likewise, if $m=k-1$ and the pole corresponding to some $C$ belongs to $L$, then $f(C\cap L)$ lies in an affine hyperplane, giving the same contradiction. Consequently either $m=k$, or $m=k-1$ and $L$ avoids every pole of $f$ corresponding to a connected component with positive $\mu$-measure. 
    There are only finitely many such components since $\supp(\mu)$ is compact. 

    Let $\iota\colon\RR^m\to L$ be an affine isometry. Each similarity $\psi_i$ generating $\mu$ preserves $L$: indeed, $\psi_i(L)$ is the smallest affine subspace containing $\psi_i(\supp(\mu))$, so $\psi_i(L)\subseteq L$, and both have dimension $m$. Therefore $\mu'=\mu_{\iota^{-1}}$ is a self-similar measure on $\RR^m$ and is affinely irreducible by Lemma~\ref{lma: prod of SS}. 

    Consider $f\circ\iota \colon \iota^{-1}(U\cap L) \to\RR^k$. In the full-dimensional case, the first assertion of Lemma~\ref{lem:highernonconical} shows that $f\circ\iota$ is non-conical. In the hyperplane case, every relevant pole lies off $L$, so its second assertion gives the same conclusion. These assertions hold on every ball in $\iota^{-1}(U\cap L)$, including when $\iota^{-1}(U\cap L)$ is disconnected. Since $\nu=\mu'_{f\circ\iota}$, Theorem~\ref{thm: maingeneraldomain} gives the required polynomial Fourier decay.
\end{proof}

\begin{rem}
Notice that, unlike the $k=2$ case, for $k\geq 3$, we do not need an analogue of the ``curve-irreducibility'' condition (which would be stronger than the affine irreducibility condition). 
This is due to the rigidity of conformal maps in $\RR^k$, $k\geq 3$, cf. Lemma~\ref{lem:highernonconical}. 
\end{rem}

\section{Open problems}\label{sec: open}

Theorem~\ref{thm: main} gives sufficient conditions for polynomial Fourier decay, but there is still some room for further work, which we keep for potential future endeavours. \medskip

    \textbf{Beyond the non-expanding condition:} Theorem~\ref{thm: maingeneraldomain} with assumption~\eqref{i:expanding} allows the self-similar measure to have arbitrary rotation group. The cost is that we have to assume $f$ is non-conical. As mentioned in Remark~\ref{rem:sharp}~\eqref{i:dreamconj}, we suspect that the weaker assumption of non-degeneracy should suffice, i.e. we still think that \cite[Conjecture~5.10]{BY-quantitative} should hold. 
    If this is indeed the case, then there would in particular be polynomial Fourier decay for pushforwards of (possibly expanding) affintely irreducible self-similar measures supported inside $\RR^k \setminus \{\mathbf{0}\}$ for $k \geq 3$ under the map $\bx \mapsto |\bx|$ whose graph in $\RR^{k+1}$ is a cone. \medskip

    \textbf{Relaxing the analyticity of $f$}: We required that $f$ is real analytic, especially for the uniform measure {\L}ojasiewicz inequality. 
    Our results are likely to hold for smooth maps whose Taylor coefficients satisfy some conditions along the lines of the `non-flatness' assumptions in \cite[Section~1.4]{ACWW25}. 
    If one only assumes non-degeneracy / non-conicality as defined in Section~\ref{s:nondeg} but relaxes the regularity from real-analytic to $C^{\infty}$, then there is no hope of a general polynomial Fourier decay result even in the $k=d=1$ case, as was recently shown in~\cite{BBslow}. \medskip

    \textbf{Quantifying the decay:} In~\cite{BY-quantitative} we obtained nontrivial Fourier decay lower bounds for images of large self-similar measures under maps $f\colon \mathbb{R}^k\to\mathbb{R}$ with rigid curvature conditions. It is likely that one could obtain certain lower bounds for general non-degenerate maps in terms of the `degree of non-degeneracy' over the support of the to-be-pushed self-similar measure, in place of the non-zero Gaussian curvature condition from~\cite{BY-quantitative}. 
    Note that in the van der Corput type result \cite[Theorem~1.1]{ACWW25} for nonlinear images $\RR \to \RR$ of self-similar measures, the lower bound for the exponent of decay that is proved does depend on the `degree of non-degeneracy' of the map. To prove such a result in higher dimensions, one cannot rely on the general non-quantitative version of the Fourier decay outside sparse frequencies property (Definition~\ref{d:afd}). \medskip

\textbf{Self-conformal measures inside curves or in $\RR^k$:} 
 In Proposition~\ref{p:conjcurve} we addressed curve-reducible self-conformal measures which are conjugate to linear, but this leaves open the question of those which are not conjugate to linear. This may require quantitative versions of the arguments from~\cite{BS23,AHW23}. 
\begin{conj}
    Every self-conformal measure on $\CC$ which is supported in an analytic curve, gives zero measure to every straight line, and is not conjugate to linear, must have polynomial Fourier decay. 
\end{conj}

Regarding self-conformal measures in $\RR^k$ for $k \geq 3$, the following seems plausible. 

\begin{conj}\label{conj: higherconformal}
    Theorem~\ref{thm: higherconformal} remains true even without the assumption that $\nu$ is conjugate to linear. 
\end{conj}
Recall that~\cite{AHW25} is about planar measures. 
Proving Conjecture~\ref{conj: higherconformal} would perhaps require spectral gap arguments (noting~\cite{AHW25,BakerKhalilSahlsten}). 
One could also consider Fourier decay for Gibbs measures for $C^1$ potentials, associated with conjugate-to-linear conformal IFSs in $\mathbb{R}^k$; in the non-conjugate case Gibbs measures were considered in~\cite{JordanSahlsten,BakerKhalilSahlsten}. \medskip

\textbf{Nonlinear fibres and conditional measures:} In \cite[Section~4]{BY-quantitative}, we discussed several results related to nonlinear projections of self-similar measures that can be treated with our Fourier decay results. However, we note here that our arguments do not give much information about fibres or conditional measures for those projections. More precisely, let $\mu$ be a probability measure on $\mathbb{R}^k$ and let $f\colon \mathbb{R}^k\to\mathbb{R}^d$ be a smooth (or analytic) map. We considered several properties of the measure $\mu_f$ in the case when $\mu$ is self-similar. Now, let $\mu^\bx$ for $\bx\in\mathbb{R}^d$ be the disintegration of $\mu$ against $\mu_f$. The conditional measures are determined only for $\mu_f$-a.e. $\bx$. At a specific exceptional $\bx$, a chosen version of the disintegration would have to be specified; the geometric fibre $f^{-1}(\bx)\cap\supp(\mu)$ is independent of this choice. Several questions can be asked; we list two of them. 

\begin{itemize}
    \item What can be said about $\dim (\supp(\mu^\bx))$ for $\mu_f$-a.e. $\bx$, and about $\dim(f^{-1}(\bx)\cap\supp(\mu))$ for specific $\bx$? Here $\dim$ can be $\Haus, \boxd$, etc. It is not hard to obtain some upper bound for $\dim (\supp(\mu^\bx))$ using results from~\cite{Sh,Wu,CorsoShmerkin}. However, much less is known about the lower bound.
    \item What can be said about the exact dimensionality of $\mu^\bx$ for $\mu_f$-a.e. $\bx$? Here, we say that a probability measure $\nu$ on $\mathbb{R}^k$ is exact dimensional if there is some $s\geq 0$ so that for $\nu$-a.e. $\bx\in\mathbb{R}^k$,
    \[
    \frac{\log \nu(B(\bx,\delta))}{\log \delta} \to s \quad \mbox{ as } \delta \to 0.
    \]
    It is known that all self-conformal~\cite{FengHu} and self-affine~\cite{FengExact} measures are exact dimensional. 
    However, much less is known about the exact dimensionality of conditional measures under smooth projections.
\end{itemize}

\section*{Acknowledgements}
\textbf{Thanks.} We thank Amir Algom, Simon Baker, Osama Khalil, Akshat Mudgal, Alex Rutar, and Tuomas Sahlsten for helpful discussions related to this work. 
We thank Ethan Eglite for pointing out an inaccuracy in a previous version of the statement of Theorem~\ref{thm: main}, and Shuqin Zhang for catching some typos. 

\textbf{Funding.} AB was financially supported by Simon Baker's EPSRC New Investigators Award (EP/W003880/1) at Loughborough University. AB was also supported by the Marie Sk\l{}odowska-Curie Actions postdoctoral fellowship FoDeNoF (no.~101210409) from the European Union, and by Tuomas Orponen's Research Council of Finland grant (no.~355453), at the University of Jyv\"askyl\"a. 
HY was financially supported by the Leverhulme Trust (ECF-2023-186) as well as Chongqing University. 

\textbf{AI use.} The first arXiv version of this manuscript was written without AI assistance, other than for the TikZ code for Figure~\ref{fig:structure} which ChatGPT~(o3) helped to create. ChatGPT~5.6~Sol was subsequently used in proofreading the first arXiv version and found several minor gaps and one more significant mistake. Qwen3.5, ChatGPT~6.1~Sol and ChatGPT~6.0~Astra were then used in assisting with the revision of some proofs such as Lemma~\ref{lma: simultaneousconjugacy}, and formatting the PDF document so that there are no overfull boxes. 

\textbf{Rights.} For the purpose of open access, the authors have applied a Creative Commons Attribution (CC-BY) licence to any Author Accepted Manuscript version arising from this submission.

\bibliographystyle{amsplain}

\end{document}